\documentclass[leqno, 12pt]{amsart}

\usepackage[top=90pt, bottom=70pt, left=70pt, right=60pt]{geometry}
\usepackage{amssymb}
\usepackage[all]{xy}
\usepackage{amsthm,enumerate}
\usepackage{hyperref}

\newtheorem{thm}{Theorem}[section]
\newtheorem{prop}[thm]{Proposition}
\newtheorem{lmm}[thm]{Lemma}
\newtheorem{cor}[thm]{Corollary}
\newtheorem{Th}{Theorem}

\newtheorem{Prop}[Th]{Proposition}

\theoremstyle{definition}
{\rm}
\newtheorem{defi}[thm]{Definition}
{\rm}
\newtheorem{Def}[thm]{Definition}{\rm}

\newcommand{\curs}{\EuR}
\newcommand{\Ab}{\curs{Ab}}

\newcommand{\Z}{\mathbb{Z}} %integers
\newcommand{\N}{\mathbb{N}} %natural numbers

\let\nsg=\normal

\newcommand{\bb}{\mathcal{B}} %categorical classifying space
\newcommand{\hh}{\mathcal{H}} %set H
\newcommand{\oo}{\mathcal{O}}
\newcommand{\pp}{\mathcal{P}} %set P
\newcommand{\ZZ}{\mathcal{Z}}

\newcommand{\calb}{\mathcal{B}}

\newcommand{\FF}{\mathcal{F}} %fusion system
\newcommand{\LL}{\mathcal{L}} %linking system
\newcommand{\hLL}{\widehat{\LL}} %reduced linking system
\newcommand{\TT}{\mathcal{T}} %transporter system
\newcommand{\g}{\mathcal{G}} %p-local group G
\newcommand{\ploc}{(S, \FF, \LL)} % G = (S, F, L)

\newcommand{\rk}{\mathrm{rk}} %rank of...

\newcommand{\G}{\curs{G}}

\newcommand{\etz}{\mathbf{T}}

\newcommand{\Hom}{\mathrm{Hom}}
\newcommand{\Inj}{\mathrm{Inj}}
\newcommand{\Map}{\mathrm{Map}}
\newcommand{\Mor}{\mathrm{Mor}}
\newcommand{\Iso}{\mathrm{Iso}}
\newcommand{\Syl}{\mathrm{Syl}}
\newcommand{\Aut}{\mathrm{Aut}}
\newcommand{\Out}{\mathrm{Out}}
\newcommand{\Inn}{\mathrm{Inn}}
\newcommand{\Rep}{\mathrm{Rep}}
\newcommand{\Ker}{\mathrm{Ker}}

\newcommand{\homf}{\Hom_\FF} %hom sets in F
\newcommand{\autf}{\Aut_\FF} %aut groups in F
\newcommand{\repf}{\Rep_\FF} %rep sets in F
\newcommand{\outf}{\Out_\FF} %out groups in F

\newcommand{\typ}{\mathrm{typ}} %isotypical stuff
\newcommand{\atyp}{\Aut_{\typ}} %isotypical automorphisms
\newcommand{\otyp}{\Out_{\typ}} %isotypical outer automorphisms

\newcommand{\fus}{\mathrm{fus}} %fusion-preserving morphisms

\newcommand{\Id}{\mathrm{Id}}
\newcommand{\incl}{\mathrm{incl}}
\newcommand{\proj}{\mathrm{proj}}

\newcommand{\Ob}{\mathrm{Ob}} %Object set
\newcommand{\grp}{\mathcal{G}r} %Groups
\newcommand{\Cat}{\curs{Cat}} %Categories
\newcommand{\pcom}{^{\wedge}_p} %p-completion functor
\newcommand{\functor}{(\underline{\phantom{B}})^{\bullet}} %bullet functor
\newcommand{\op}{\mathrm{op}} %opposite category

\newcommand{\ptor}{\Z/p^{\infty}} %p-Pruffer group
\newcommand{\sylp}{\mathrm{Syl}_p}

\newcommand{\conj}[2]{{#1}^{#2}} %conjugacy class
\newcommand{\higherlim}[2]{\displaystyle\setbox1=\hbox{\rm lim}
	\setbox2=\hbox to \wd1{\leftarrowfill} \ht2=0pt \dp2=-1pt
	\setbox3=\hbox{$\scriptstyle{#1}$}
	\def\test{#1}\ifx\test\empty
	\mathop{\mathop{\vtop{\baselineskip=5pt\box1\box2}}}\nolimits^{#2}
	\else
	\ifdim\wd1<\wd3
	\mathop{\hphantom{^{#2}}\vtop{\baselineskip=5pt\box1\box2}^{#2}}_{#1}
	\else
	\mathop{\mathop{\vtop{\baselineskip=5pt\box1\box2}}_{#1}}%
	\nolimits^{#2}
	\fi\fi}
\newcommand{\invlim}[1]{\higherlim{#1}{}}

\newcommand{\defin}{\stackrel{\mathrm{def}} = } %definition
\def\colim{\mathop{\rm colim}} %colimit
\def\hocolim{\mathop{\rm hocolim}} %homotopy colimit
\def\defeq{\overset{\operatorname{def}}{=}}

\newcommand{\hadgesh}[1]{\textcolor{black}{\emph{#1}}}

\DeclareMathAlphabet\EuR{U}{eur}{m}{n}
\SetMathAlphabet\EuR{bold}{U}{eur}{b}{n}

\newcommand{\xxto}[1]{\mathrel{\mathop{%
  \setbox0\hbox{$\ {\scriptstyle#1}\ $}%
  \hbox to \wd0{\rightarrowfill}}^{#1}}%
}

\newcommand{\xto}[2][]{%
  \mathrel{\mathop{%
    \setbox0\vbox{%%\mathsurround=0pt
      \hbox{$\scriptstyle\;\;{#1}\;\;$}%
      \hbox{$\scriptstyle\;\;{#2}\;\;$}%
    }%
    \hbox to\wd0{\rightarrowfill}\displaystyle}%
  \limits^{#2}\ifx{#1}{}\else{_{#1}}\fi}%
}

\newcommand{\xlto}[2][]{%
  \mathrel{\mathop{%
    \setbox0\vbox{%%\mathsurround=0pt
      \hbox{$\scriptstyle\;\;{#1}\;\;$}%
      \hbox{$\scriptstyle\;\;{#2}\;\;$}%
    }%
    \hbox to\wd0{\leftarrowfill}\displaystyle}%
  \limits^{#2}\ifx{#1}{}\else{_{#1}}\fi}%
}

\newcommand{\longleft}[1]{\;{\leftarrow%
\count255=0 \loop \mathrel{\mkern-6mu}%
    \relbar\advance\count255 by1\ifnum\count255<#1\repeat}\;}
\newcommand{\longright}[1]{\;{\count255=0 \loop \relbar\mathrel{\mkern-6mu}%
    \advance\count255 by1\ifnum\count255<#1\repeat\rightarrow}\;}
\newcommand{\Right}[2]{\overset{#2}{\longright#1}}
\newcommand{\RIGHT}[3]{\mathrel{\mathop{\kern0pt\longright#1}
        \limits^{#2}_{#3}}}
\newcommand{\Left}[2]{{\buildrel #2 \over {\longleft#1}}}
\newcommand{\LEFT}[3]{\mathrel{\mathop{\kern0pt\longleft#1}\limits^{#2}_{#3}}
}
\newcommand{\dRIGHT}[3]{\mathrel{%
   \mathop{\vcenter{\baselineskip=0pt\hbox{$\kern0pt\longright#1$}%
   \hbox{$\kern0pt\longright#1$}}}\limits^{#2}_{#3}}}
\newcommand{\LRIGHT}[3]{\mathrel{%
   \mathop{\vcenter{\baselineskip=0pt\hbox{$\kern0pt\longleft#1$}%
   \hbox{$\kern0pt\longright#1$}}}\limits^{#2}_{#3}}}
\newcommand{\RLEFT}[3]{\mathrel{%
   \mathop{\vcenter{\baselineskip=0pt\hbox{$\kern0pt\longright#1$}%
   \hbox{$\kern0pt\longleft#1$}}}\limits^{#2}_{#3}}}
\newcommand{\onto}[1]{\;{\count255=0 \loop \relbar\joinrel
    \advance\count255 by1
    \ifnum\count255<#1 \repeat \twoheadrightarrow}\;}

\usepackage[usenames]{color}

\begin{document}

\title{Automorphisms of $p$-local compact groups}
\author{A. Gonz\'alez, R. Levi}
%\date{}
%\address{}
%\email{}
%\thanks{}

\begin{abstract}
Self equivalences of classifying spaces of $p$-local compact groups are well understood by means of the algebraic structure that gives rise to them, but explicit descriptions are lacking. In this paper we use Robinson's construction of an amalgam $G$, realising a given fusion system, to produce a split epimorphism from the outer automorphism group of $G$ to the group of homotopy classes of self homotopy equivalences of the classifying space of the corresponding $p$-local compact group.\end{abstract}

\maketitle

%\tableofcontents

A $p$-local compact group is an algebraic object which is modelled on the $p$-local homotopy theory of classifying spaces of compact Lie groups and $p$-compact groups. These objects were constructed by Broto, Levi and Oliver in \cite{BLO2, BLO3}, and have been studied in an increasing number of papers since. Many aspects of the theory are still wide open, and among such aspects, a reasonable notion of maps between $p$-local compact groups remains quite elusive. Self equivalences of $p$-local compact groups are relatively well understood in terms of certain automorphisms of the underlying Sylow subgroup, but computing these groups of automorphisms is generally out of reach. This paper offers a way of regarding such automorphisms as being induced, in the appropriate sense, by automorphisms of certain discrete groups associated to that, in the finite case, were introduced by G. Robinson in \cite{Robinson}. The  construction will be generalised here to the compact analog. We will refer to such groups as \hadgesh{Robinson groups}

Before  stating our results, we briefly recall the terminology, with a more detailed discussion in the next section. A $p$-local compact group is a triple $\g = \ploc$, where $S$ is a discrete $p$-toral group (to be defined later), and $\FF$ and $\LL$ are categories whose objects are certain subgroups of $S$. The classifying space of $\g$, which we denote by $B\g$ is the $p$-completed nerve $|\LL|\pcom$.
For any space $X$, let $\Out(X)$ denote the group of homotopy classes of self homotopy equivalences of $X$. We are now ready to state our main result.

\begin{Th}\label{main}
Let $\g = \ploc$ be a $p$-local compact group. Then there exist a group $G$, with a map $\mu\colon BG\to B\g$, and a a split epimorphism
\[\omega \colon \Out(G) \to \Out(B\g).\]
Furthermore, the following statements hold:
\begin{enumerate}[(i)]
\item If $[\Psi]\in\Out(B\g)$ and $[\Phi]\in\omega^{-1}([\Psi])$ are represented by $\Psi$ and $\Phi$ respectively, then $\mu\circ B\Phi \simeq \Psi \circ\mu$.
\item Let $\Aut(G,1_S)\le \Aut(G)$ denote the subgroup of automorphisms which restrict to the identity on $S$, and let $\Out(G,1_S)$ denote its quotient by the subgroup of inner automorphisms induced by conjugation by elements of $C_G(S)$. Then $\Ker(\omega) \cong \Out(G, 1_S)$ if $p\neq 2$, and for $p=2$ there is a short exact sequence,
\[1\to \Ker(\omega) \to \Out(G, 1_S)\to \higherlim{\oo(\FF^c)}{1}\ZZ/\ZZ_0\to 1.\]
\end{enumerate}
\end{Th}

The group $G$ in the theorem is a specific instance of the construction given in \cite{LS} of a group that realises the fusion system $\FF$, which in turn is a quotient group of Robinson's construction on the same data. We shall discuss these constructions in more detail in Section 2.
 
In some particular cases the statement Theorem \ref{main} can be strengthened. 

\begin{Prop}\label{PropB}
Let $\g = \ploc$ be a $p$-local compact group, and let $G$ be the model for $\FF$ in Theorem \ref{main}. If $\TT^c_S(G) \cong \LL$, then the epimorphism $\omega \colon \Out(G) \to \Out(B\g)$ is an isomorphism.
\end{Prop}

The paper is organised as follows. In Section 1 we review the theory of $p$-local compact groups and some related notions, like a topological construction of fusion systems, originally introduced in \cite{BLO2}. Section 2 contains the construction of the Robinson amalgam realising a given (saturated) fusion system. In this section we also include a brief discussion about graphs of groups, with particular emphasis on trees of groups, and some topological consequences of these constructions. Section 3 contains the proofs of \ref{main}. Section 4 contains the proof of Proposition \ref{PropB}, together with some specific examples.

\bigskip

The authors would like to thank  A. Chermak and G. Robinson for useful conversations. 

%%%%%%%%%%%%%%%%%%%%%%%%%%%%%%%%%%%%%%%%%%%%
%%%%%%%%%%%%%%%%%%%%%%%%%%%%%%%%%%%%%%%%%%%%
%%%%%% SECTION 1 %%%%%%%%%%%%%%%%%%%%%%%%%%%%%%%
%%%%%%%%%%%%%%%%%%%%%%%%%%%%%%%%%%%%%%%%%%%%
%%%%%%%%%%%%%%%%%%%%%%%%%%%%%%%%%%%%%%%%%%%%
%%%%%%%%%%%%%%%%%%%%%%%%%%%%%%%%%%%%%%%%%%%%

\section{Background on $p$-local compact groups}

In this section we recall the basic concepts of $p$-local compact group theory, as well as constructions and facts to be used throughout the paper. For a detailed discussion of $p$-local compact groups the reader is referred to \cite{BLO3}.

\begin{Def}\label{defidiscreteptoral}

Let $\Z/p^\infty$ denote the union of all $\Z/p^n$ under the obvious inclusions.
A \hadgesh{discrete $p$-torus} is a group $T$ isomorphic to a finite direct product  $(\ptor)^{\times r}$.  The number $r$ is called the \hadgesh{rank} of $T$. A \hadgesh{discrete $p$-toral group} is a group $P$, which contains a discrete $p$-torus $T$ as a normal subgroup of $p$-power index.  The subgroup $T$ is referred to as the \hadgesh{maximal torus} of $P$,  and the \hadgesh{rank} of $P$ is defined to be the rank of its maximal torus.
\end{Def}

The groups we work with in this paper are generally infinite, but contain a unique conjugacy class of a maximal discrete $p$-toral subgroups.

\begin{Def}\label{Sylow}
Let $G$ be any discrete group. We say that a discrete $p$-toral group $S\le G$ is \hadgesh{a Sylow $p$-subgroup of $G$} if any other discrete $p$-toral subgroup $P\le G$ is conjugate in $G$ to a subgroup of $S$.
\end{Def}

We now recall the definition of a fusion system over a discrete $p$-toral group.
 
\begin{Def}\label{defifusion}

A \hadgesh{fusion system} $\FF$ over a discrete $p$-toral group $S$ is a category whose objects are the subgroups of $S$ and whose morphism sets $\homf(P,P')$ satisfy the following conditions:
\begin{enumerate}[(i)]

\item $\Hom_S(P,P') \subseteq \homf(P,P') \subseteq \Inj(P,P')$ for all $P,P' \leq S$.

\item Every morphism in $\FF$ factors as an isomorphism in $\FF$ followed by an inclusion.

\end{enumerate}

\end{Def}

Given a fusion system $\FF$ over a discrete $p$-toral group $S$, we will often refer to $S_0$ also as the \hadgesh{maximal torus} of $\FF$, and the \hadgesh{rank of $\FF$} will always mean the rank of $S$. Two subgroups $P,P'$ are called \hadgesh{$\FF$-conjugate} if $\Iso_{\FF}(P,P') \neq \emptyset$. For a subgroup $P \leq S$, we denote 
$$
P^\FF = \{P' \leq S | P'  \mbox{ is } \FF \mbox{-conjugate to } P\}.
$$

Define the \hadgesh{order} of a discrete $p$-toral group $P$   to be the pair $|P| \defin (\rk(P), |P/P_0|)$, as in \cite{BLO3}. Thus, given two discrete $p$-toral groups $P$ and $Q$ we say that $|P| \leq |Q|$ if either $\rk(P) < \rk(Q)$, or $\rk(P) = \rk(Q)$ and $|P/P_0| \leq |Q/Q_0|$.

\begin{Def}\label{defifcfn}

Let $\FF$ be a fusion system over a discrete $p$-toral group $S$.  A subgroup $P \leq S$ is said to be 

\begin{itemize}

\item \hadgesh{fully $\FF$-centralised}, if  $|C_S(P')| \leq |C_S(P)|$, for all $P' \in P^\FF$.

\item  \hadgesh{fully $\FF$-normalised}, if $|N_S(P')| \leq |N_S(P)|$ for all $P'\in P^\FF$.

\end{itemize}

\end{Def}

\begin{Def}\label{defisaturation}

A fusion system $\FF$ over $S$ is said to be \hadgesh{saturated} if the following three conditions hold:
\begin{enumerate}[(I)]

\item For each $P \leq S$ which is fully $\FF$-normalised, $P$ is fully $\FF$-centralised, $\Out_{\FF}(P)$ is finite and $\Out_S(P) \in \Syl_p(\Out_{\FF}(P))$.

\item If $P \leq S$ and $f \in \homf(P,S)$ is such that $P' = f(P)$ is fully $\FF$-centralised, then there exists $\widetilde{f} \in \homf(N_f,S)$ such that $f = \widetilde{f}_{|P}$, where
$$
N_f = \{g \in N_S(P) | f \circ c_g \circ f^{-1} \in \Aut_S(P')\}.
$$

\item If $P_1 \leq P_2 \leq P_3 \leq \ldots$ is an increasing sequence of subgroups of $S$, with $P = \cup_{n=1}^{\infty} P_n$, and if $f \in \Hom(P,S)$ is any homomorphism such that $f_{|P_n} \in \Hom_{\FF}(P_n,S)$ for all $n$, then $f \in \Hom_{\FF}(P,S)$.

\end{enumerate}

\end{Def}

Notice that for a fusion system $\FF$ over $S$ and any subgroups $P,Q\le S$, the group $\Inn(Q)$ is contained in $\autf(Q)$ and acts on $\homf(P,Q)$ by left composition. Let
\[\repf(P,Q) \defin \Inn(Q)\backslash\homf(P,Q)\quad\text{and}\quad \outf(Q) \defin \Inn(Q)\backslash\autf(Q).\]
By definition if $\FF$ is a saturated fusion system over $S$, then for all $P \leq S$, $\Aut_{\FF}(P)$ is an artinian, locally finite group and has a Sylow $p$-subgroups of finite index.

\begin{Def}\label{deficentricrad}

If $G$ is a group and $H\le G$, we say that $H$ is \textit{centric in $G$}, or $G$-centric, if $C_G(H) = Z(H)$. Let $\FF$ be a fusion system over a discrete $p$-toral group. A subgroup $P \leq S$ is said to be \hadgesh{$\FF$-centric} if every $P'\in P^\FF$ is centric in $S$. A subgroup $P \leq S$ is called \hadgesh{$\FF$-radical} if $\outf(P)$ contains no nontrivial normal $p$-subgroup.

\end{Def}

\newcommand{\HH}{\mathcal{H}}

Clearly, $\FF$-centric subgroups are fully $\FF$-centralised, and conversely, if $P$ is fully $\FF$-centralised and centric in $S$, then it is $\FF$-centric. Next we recall the definition of transporter system, as introduced in \cite{BLO6}. The concept of centric linking system is a particular case, specified at the end of the following definition. For a group $G$ and a collection $\mathcal{H}$ of subgroups of $G$, the transporter category of $G$ with respect to $\mathcal{H}$ is the category $\TT_{\mathcal{H}}(G)$ whose objects are the set $\mathcal{H}$,  and whose morphism sets are $\Mor_{\TT_{\mathcal{H}}(G)}(P,Q) = N_G(P,Q)$, the collection of all elements of $G$ which conjugates $P$ into $Q$.

\begin{defi}\label{defitransporter}

Let $\FF$ be a fusion system over a finite $p$-group $S$. A \textit{transporter system} associated to $\FF$ is a category $\TT$ such that $\Ob(\TT) \subseteq \Ob(\FF)$, together with a couple of functors
$$
\TT_{\Ob(\TT)}(S) \stackrel{\varepsilon} \longrightarrow \TT \stackrel{\rho} \longrightarrow \FF,
$$
satisfying the following axioms:
\begin{itemize}
 \item[(A1)] $\Ob(\TT)$ is closed under $\FF$-conjugacy and overgroups. Also, $\varepsilon$ is the identity on objects and $\rho$ is inclusion on objects.

 \item[(A2)] For each $P \in \Ob(\TT)$, let
$$
E(P) = \Ker(\Aut_{\TT}(P) \to \Aut_{\FF}(P)).
$$
Then, for each $P, P' \in \Ob(\TT)$, $E(P)$ acts freely on $\Mor_{\TT}(P, P')$ by right composition, and $\rho_{P, P'}$ is the orbit map for this action. Also, $E(P')$ acts freely on $\Mor_{\TT}(P, P')$ by left composition.

 \item[(B)] For each $P, P' \in \Ob(\TT)$, $\varepsilon_{P,P'}: N_S(P,P') \to \Mor_{\TT}(P, P')$ is injective, and the composite $\rho_{P,P'} \circ \varepsilon_{P,P'}$ sends $g \in N_S(P, P')$ to $c_g \in \Hom_{\FF}(P,P')$.

 \item[(C)] For all $\varphi \in \Mor_{\TT}(P,P')$ and all $g \in P$, the following diagram commutes in $\TT$:
$$
\xymatrix{
P \ar[r]^{\varphi} \ar[d]_{\varepsilon_{P,P}(g)} & P' \ar[d]^{\varepsilon_{P',P'}(\rho(\varphi)(g))} \\
P \ar[r]_{\varphi} & P'
}
$$

 \item[(I)] $\varepsilon_{S,S}(S) \in Syl_p(\Aut_{\TT}(S))$.

 \item[(II)] Let $\varphi \in \Iso_{\TT}(P,P')$, and $P \nsg R \leq S$, $P' \nsg R' \leq S$ such that
$$
\varphi \circ \varepsilon_{P,P}(R) \circ \varphi^{-1} \leq \varepsilon_{P',P'}(R').
$$
Then, there is some $\widetilde{\varphi} \in \Mor_{\TT}(R, R')$ such that $\widetilde{\varphi} \circ \varepsilon_{P,R}(1) = \varepsilon_{P', R'}(1) \circ \varphi$, that is, the following diagram is commutative in $\TT$:
$$
\xymatrix{
P \ar[r]^{\varphi} \ar[d]_{\varepsilon_{P,R}(1)} & P' \ar[d]^{\varepsilon_{P',R'}(1)} \\
R \ar[r]_{\widetilde{\varepsilon}} & R'
}
$$

\item[(III)] Assume $P_1 \leq P_2 \leq \ldots$ in $\Ob(\TT)$ and $\varphi_n \in \Mor_{\TT}(P_n, S)$ are such that, for all $n \geq 1$, $\varphi_n = \varphi_{n+1} \circ \varepsilon_{P_n, P_{n+1}}(1)$. Set $P = \bigcup_{n = 1}^{\infty} P_n$. Then there is $\varphi \in \Mor_{\TT}(P,S)$ such that $\varphi_n = \varphi \circ \varepsilon_{P_n,P}(1)$ for all $n \geq 1$.

\end{itemize}
Given a transporter system $\TT$, the \textit{classifying space} of $\TT$ is the space $B\TT \defeq |\TT|\pcom$.

A \hadgesh{centric linking system} is a transporter system $\LL$ whose object set is the collection of all $\FF$-centric subgroups, and such that $E(P) = \varepsilon(Z(P))$ for all $P \in \Ob(\LL)$.

A \hadgesh{$p$-local compact group} is a triple $\g = (S, \FF, \LL)$, where $S$ is a discrete $p$-toral group, $\FF$ is a saturated fusion system over $S$, and $\LL$ is a centric linking system associated to $\FF$. Given a $p$-local compact group $\g$, the subgroup $S_0 \leq S$ will be called the \hadgesh{maximal torus} of $\g$, and the \hadgesh{rank} of $\g$ is defined to be $\rk(S)$.

\end{defi}

We will, in general, denote a $p$-local compact group just by $\g$, and unless otherwise specified, will refer to $S$, $\FF$ and $\LL$ as its Sylow subgroup, fusion system and linking system respectively. The following result about existence and uniqueness of centric linking systems is due to Levi and Libman, see \cite[Theorem B]{LL}.

\begin{cor}

Let $\FF$ be a saturated fusion system over a discrete $p$-toral group $S$. Then, up to isomorphism there exists a unique centric linking system associated to $\FF$.

\end{cor}

Next we revise the very useful \hadgesh{``bullet construction"}. For a full account the reader is referred to \cite[Sec. 3]{BLO3}.

\begin{prop}\label{blo3-S3}

Let $\g$ be a $p$-local compact group. Then, there are functors $\functor_{\FF} : \FF \to \FF$ and  $\functor_\LL: \LL \to \LL$,  satisfying the following properties:

\begin{enumerate}[\rm(a)]

\item $\functor_\FF\circ\rho = \rho\circ\functor_\LL$.

\item For each $P\le S$, $P\le P^\bullet\defin (P)^\bullet_\FF$, and if $P\in\LL$, then $P^\bullet=(P)^\bullet_\LL$.

\item Let $\FF^{\bullet}$ be the full subcategory of $\FF$, with object set $\{P^{\bullet} \mbox{ } | \mbox{ } P \in \Ob(\FF)\}$. Then $\FF^\bullet$ contains finitely many $S$-conjugacy classes of subgroups, and contains all the $\FF$-centric, $\FF$-radical subgroups of $S$.

\item If $P$ is $\FF$-centric then $Z(P^{\bullet}) = Z(P)$.

\end{enumerate}
Furthermore, if we set $\LL^{\bullet} \subseteq \LL$ for the full subcategory with objects in $\Ob(\FF^{\bullet}) \cap \Ob(\LL)$, then the inclusion is right adjoint to $\functor: \LL \to \LL^{\bullet}$, and hence $|\LL^{\bullet}| \simeq |\LL|$.
\end{prop}

%%%%%%%%%%%%

\subsection{Automorphisms of $p$-local compact groups}

The space of self equivalences of $B\g$, for a $p$-local compact group $\g$, can be described completely in terms of certain self equivalences of the associated centric linking system, as shown in  \cite[Section 7]{BLO3}. We summarise here the necessary definitions and results.

\begin{defi}\label{isotyp}

Let $\g$ be a $p$-local compact group. Let $\hh$ be a collection of objects in $\LL$, and let $\LL_\hh$ denote the full subcategory with objects in $\hh$.  Let $\varepsilon: \TT_{\hh}(S) \to \LL_\hh$ be the restriction of the structure functor, as in Definition \ref{defitransporter}, to $\LL_\hh$. A functor $\Psi: \LL_\hh \to \LL_\hh$ is said to be an \textit{isotypical equivalence} of $\LL_\hh$ if $\Psi$ is an equivalence of categories which restricts to an equivalence on $\TT_{\hh}(S)$. 

Let $\atyp(\LL_\hh)$ denote the groupoid whose objects are all isotypical self equivalences of $\LL_\hh$, and whose morphisms are natural isomorphisms between them. Let $\otyp(\LL_\hh)$ denote the group of components of $\atyp(\LL_\hh)$, namely the group obtained by identifying two isotypical equivalences if they are naturally isomorphic.

\end{defi}

\newcommand{\SFL}{(S, \FF,\LL)}

We will sometimes  refer to isotypical self equivalence of $\LL$ simply as \textit{automorphisms of $\g$}. 

\begin{defi}\label{aut^I}
Let $\g=\SFL$ be a $p$-local compact group. For each $P, Q\in \LL$ such that $P\le Q$, set $\iota_{P,Q} = \delta_{P,Q}(1)\in\Mor_\LL(P,Q)$, and let $I = I_{\LL} \defin \{\iota_{P, Q}\;|\; P,Q\in\LL\;, P\le Q\}$. Finally, let $\atyp^I(\LL) \subseteq \atyp(\LL)$ be the submonoid of isotypical equivalences which leave $I$ invariant. 
\end{defi}

\begin{lmm}\label{AOV-1.14}
Let $\g=\SFL$ be a $p$-local compact group. Then,
\begin{enumerate}[{\rm(a)}]

\item Every element in $\atyp^I(\LL)$ is invertible. Hence, $\atyp^I(\LL)$ is a group.

\item Every $\Psi \in \atyp(\LL)$ is naturally isomorphic to some $\Psi' \in \atyp^I(\LL)$.

\item The group $\otyp(\LL)$ is isomorphic to the quotient of $\atyp^I(\LL)$ by the obvious conjugation action of $\Aut_\LL(S)$.

\end{enumerate}

\end{lmm}

\begin{proof}

This is proved in   \cite[Lemma 1.14]{AOV}.
\end{proof}

\begin{prop}\label{autext}

Let $\g = \ploc$ be a $p$-local compact group. Let also $\Psi$ be an isotypical equivalence of the full subcategory $\LL^{\bullet} \subseteq \LL$ that respects inclusions (i.e., $\Psi(\iota_{P,Q}) = \iota_{\Psi(P), \Psi(Q)}$ for all $P, Q \in \Ob(\LL^{\bullet})$). Then there exists a unique $\widetilde{\Psi} \in \Aut_{\typ}^{I}(\LL)$ extending $\Psi$.

\end{prop}

\begin{proof}

This is proved in \cite[Proposition 1.14]{JLL}.
\end{proof}

An important consequence of $p$-local group theory is that for any $p$-local compact group $\g$, the topological monoid $\Aut(B\g)$ of self homotopy equivalences of $B\g$ is determined by isotypical equivalences of $\LL$.

\begin{thm}[{\cite[Theorem 7.1]{BLO3}}]

Let $\g=\ploc$ be a $p$-local compact group. For each $i \geq 1$, let $\pi_i(B\ZZ^{\wedge}_p): \oo^c(\FF)^{op} \to \Ab$ be the functor that sends $P$ to $\pi_i(BZ(P)^{\wedge}_p)$. Then
$$
\Out(B\g)\defin \pi_0(\Aut(B\g)) \cong \otyp(\LL).
$$
Furthermore, 
\[\pi_i(\Aut(B\g)) \cong \lim_{\oo^c(\FF)} \pi_i(B\ZZ^{\wedge}_p)\] for $i = 1, 2$, and $\pi_i(\Aut(B\g)) = 0$ for $i \geq 3$.

\end{thm}

%%%%%%%%%%%%

\subsection{Constructing fusion systems}\label{ConstructFL}

We next recall two constructions through which one can obtain a fusion system over a given discrete $p$-toral group. 

For any group $G$ and $P,Q\le G$, let $N_G(P,Q) = \{g \in G \mbox{ } | \mbox{ } gPg^{-1} \leq Q\}$. The sets $N_G(P,Q)$ are called transporter sets. Let  $S$ be a discrete $p$-toral group, and let $G$ be a  group containing $S$. The \hadgesh{transporter category $\TT_S(G)$ of $G$ over $S$} is a category whose objects are the subgroups  $P\le S$, and whose morphism sets are the transporter sets
\begin{equation}\label{TSG}
\Mor_{\TT_S(G)}(P,Q) \defin N_G(P,Q)
\end{equation}
for all $P, Q \leq S$. Composition in $\TT_S(G)$ is given by multiplication.

The same setup gives rise to a fusion system over $S$. Given a group $G$ and a discrete $p$-toral subgroup $S\le G$, let   $\FF_S(G)$ denote the category whose objects are the subgroups of $S$,   and where morphism sets are defined for all $P, Q \leq S$ by
\begin{equation}\label{FSG}
\Hom_{\FF_S(G)}(P,Q) \defin N_G(P,Q)/C_G(P).
\end{equation}
Clearly $\FF_S(G)$ is a quotient category of $\TT_S(G)$. The category $\FF_S(G)$ is always a fusion system, but it is not always saturated.

If $G$ is a finite group and $S \in \Syl_p(G)$, the category $\FF=\FF_S(G)$ is always a saturated fusion system over $S$. The $\FF$-centric subgroups are those subgroups $P\le S$, such that $Z(P)\in\Syl_p(C_G(P))$. The associated centric linking system $\LL=\LL^c_S(G)$ is the category whose objects are the $\FF$-centric subgroups $P\le S$, and whose morphisms are given by $\Mor_\LL(P,Q) = N_G(P,Q)/C_G'(P)$, where $C_G'(P)$ is the normal complement of $Z(P)$ in $C_G(P)$ (see \cite{BLO2}). 

\bigskip

The second construction we review here is of topological nature (originally introduced in  \cite{BLO1}). Let $X$ be any space, let $S$ be a discrete $p$-toral group, and let $\gamma: BS \to X$ be a map. Define  $\FF_S(\gamma)$  to be the fusion system over $S$, with morphism sets
\begin{equation}\label{FSgamma}
\Hom_{\FF_S(\gamma)}(P,Q) \defin \{f \in \Inj(P,Q) \mbox{ } | \mbox{ } \gamma_{|BP} \simeq \gamma_{|BQ} \circ Bf\}
\end{equation}
for all $P,Q \leq S$. If $\g=\ploc$ is a $p$-local finite group, and $\gamma\colon BS\to B\g$ is the obvious inclusion of Sylow subgroup, then $\FF_S(\gamma)\cong \FF$, \cite[Proposition 7.3(a)]{BLO2}.

Let $\LL_S(\gamma)$ denote the category  whose objects are again the subgroups of $S$, and  for all $P,Q \leq S$, let $\Mor_{\LL_S(\gamma)}(P,Q)$ be the set of all pairs $(f,[H])$, such that $f \in \Hom_{\FF_S(\gamma)}(P,Q)$ and  $[H]$ is the path homotopy class of   a homotopy $H$  from $\gamma_{|BP}$ to $\gamma_{|BQ} \circ Bf$ in $\Map(BP, X)$. There is an obvious way to define composition, given explicitly in \cite[Definition 2.5]{BLO1}.

For a fixed map $\gamma\colon BS\to X$, the two constructions are related by an obvious projection functor, $\rho \colon \LL_S(\gamma)\to\FF_S(\gamma)$, which is the identity on objects, and sends a morphism $(f,[H])$ to $f$. These constructions are functorial: given a map $\omega\colon X  \to Y$, there are induced functors
\begin{equation}\label{induced functors}
\omega_\FF\colon\FF_S(\gamma) \to \FF_S(\omega \circ \gamma),\quad\mathrm{and}\quad \omega_\LL\colon\LL_S(\gamma) \to \LL_S(\omega \circ \gamma),
\end{equation}
and $\omega_\FF \circ\rho= \rho\circ\omega_\LL$.

\begin{defi}\label{reducedL}

Let $\g = \ploc$ be a $p$-local compact group. Then, the \textit{reduced linking system} of $\g$ is the quotient category $\hLL$ defined by $\Ob(\hLL) = \Ob(\LL)$ and by
$$
\Mor_{\hLL}(P,Q) = \Mor_{\LL}(P,Q)/(Z(P)_0)
$$
for each pair $P,Q \in \Ob(\hLL)$.

\end{defi}

Notice that if $\g$ is a $p$-local finite group, then $\hLL = \LL$, since the maximal torus of $S$ is trivial. By definition, a centric linking system determines a reduced linking system, but the converse is also true.

\begin{prop}[Proposition 6.4 \cite{BLO}]

Let $\FF$ be a fusion system over the discrete $p$-toral group $S$. Then any reduced linking system $\hLL$ associated to $\FF$ lifts to a centric linking system $\LL$ associated to $\FF$ which is unique up to isomorphism.

\end{prop}

For $\gamma\colon BS\to X$, let $\LL^c_S(\gamma)$ denote the full subcategory of $\LL_S(\gamma)$ whose objects are the $\FF_S(\gamma)$-centric subgroups of $S$.

\begin{prop}\label{reconstruct}
Let $\g$ be a $p$-local compact group, and let $\gamma: BS \to B\g$ be the natural inclusion. Then, $\FF_S(\gamma) \cong \FF$ and $\LL^c_S(\gamma) = \hLL$.
\end{prop}
\begin{proof} 

First we show that $\FF_S(\gamma) = \FF$. By \cite[Theorem 6.3 (a)]{BLO3}, 
\[ [BQ, B\g]\cong \Rep(Q, \LL) \defin \Hom(Q,S)/\sim,\]
where $\sim$ is the equivalence relation $\rho\sim\rho'$ if there is some $\alpha\in\Hom_\FF(\rho(Q),\rho'(Q))$ such that $\rho' = \alpha\circ\rho$.
 and hence
$$
\begin{array}{rll}
\Mor_{\FF_S(\gamma)}(P,Q) & \defin \{f \in \Inj(P,Q) \mbox{ } | \mbox{ } \gamma_{|BP} \simeq \gamma_{|BQ} \circ Bf \} = & \\
 & = \{f \in \Inj(P,Q) \mbox{ } | \mbox{ } \exists\, \omega \in \Hom_{\FF}(P,Q),\; \omega\circ\iota_P = \iota_Q \circ f\} = \Hom_{\FF}(P,Q),
\end{array}
$$
where $\iota_P$ and $\iota_Q$ are the respective inclusions into $S$, and the last equality holds because $\iota_Q$ is a monomorphism. Hence $\FF_S(\gamma) \cong \FF$. In particular, a subgroup $P \leq S$ is $\FF_S(\gamma)$-centric if and only if it is $\FF$-centric, and $\Ob(\LL^c_S(\gamma)) = \Ob(\hLL)$. 

Let  $\varepsilon_{\LL}: \LL \to \LL^c_S(\gamma)$ be the functor which is identity on objects and which sends a morphism $\varphi: P \to Q$ to the pair $(\pi(\varphi), [|\eta_{\varphi}|])$, where $\eta_{\varphi}$ is the natural transformation from the inclusion $\bb P \to \LL$ to the composite $\bb(P) \xto{\varphi} \bb(Q) \to \LL$. Here for any group $G$, $\bb(G)$ is the one object category with $G$ as an automorphism group of this object.   For each pair $P, Q \in \Ob(\LL)$ there is a commutative diagram
$$
\xymatrix{
\Mor_{\LL}(P,Q) \ar[r]^{\varepsilon_{\LL}} \ar[d]_{\pi} & \Mor_{\LL^c_S(\gamma)}(P,Q) \ar[d]^{\pi} \\
\Hom_{\FF}(P,Q) \ar[r]_{\cong} & \Hom_{\FF_S(\gamma)}(P,Q),\\
}
$$
where $\pi$ on the right column sends $(f, [H])$ to $f$. Since $P$ is $\FF$-centric, \[\pi_1(\Map(BP, B\g)_{\gamma_{|BP}}) \cong Z(P)/Z(P)_0\] by \cite[Theorem 6.3 (b)]{BLO3}, and hence 
\[\pi\colon \Mor_{\LL^c_S(\gamma)}(P,Q)\to \Hom_{\FF_S(\gamma)}(P,Q)\] is the orbit map of a free action of $Z(P)/Z(P)_0$ on the source. Since $\varepsilon_{\LL}$ is $Z(P)$-equivariant, it follows that $\Mor_{\LL^c_S(\gamma)}(P,Q) \cong \Mor_{\LL}(P,Q)/Z(P)_0 = \Mor_{\hLL}(P,Q)$. 
\end{proof}

When $G$ is a discrete group and $S \leq G$ is a discrete $p$-toral subgroup, then the algebraic and topological constructions are strongly related, as the following proposition states.

\begin{prop}[{\cite[Prop. 2.1]{LS}}] \label{LS-2.1}
Let $G$ be a discrete group, and let $S \leq G$ be a discrete $p$-toral subgroup. Let $\iota\colon S\to G$ denote the inclusion, and let $\gamma \defin B\iota$.  Then $\FF_S(\gamma) \cong \FF_S(G)$ and $\LL_S(\gamma) = \TT_S(G)$.
\end{prop}

%%%%%%%%%%%%%

\subsection{Constrained $p$-local compact groups}

We now consider \hadgesh{constrained} $p$-local compact groups. This generalises the same concept for $p$-local finite groups as defined in \cite{BCGLO1}.

Given a saturated fusion system $\FF$ over a discrete $p$-toral group $S$, and a fully $\FF$-normalised subgroup $P\le S$, the normaliser fusion systems $N_\FF(P)$ is defined as the fusion subsystem over $N_S(P)$ generated by morphisms in $\FF$ that restrict to automorphisms of $P$. It is shown to be saturated in \cite[Definition 2.1, Theorem 2.3]{BLO6}.

\begin{defi}\label{Fnormal}
Let  $\FF$ be a saturated fusion system over a discrete $p$-toral group $S$, and let $P \leq S$ be a subgroup. We say that $P$ is \textit{normal in $\FF$}  if $N_\FF(P) = \FF$. 
\end{defi}

\begin{defi}\label{constrained}

Let $\FF$ be a saturated fusion system over a discrete $p$-toral group $S$. Then $\FF$ is said to be \textit{constrained} if it contains an object $P$ which is $\FF$-centric and normal in $\FF$.  A \hadgesh{model for the fusion system $\FF$} is a group $G$ that is locally finite, artinian, $p'$-reduced, and such that $\FF \cong \FF_S(G)$ for some choice of Sylow $p$-subgroup $S \leq G$.
 
\end{defi}

The following proposition is a generalisation of \cite[Proposition 4.3]{BCGLO1} to fusion systems over discrete $p$-toral groups. The proof was given by the first author in \cite[\S 2]{Gonza}, and also appears  in \cite[Proposition 1.18]{LL}. 

\begin{prop}\label{constrained2}

Let $\g = \ploc$ be a $p$-local compact group, and suppose that $\FF$ is constrained. Then, $\FF$ has a unique (up to isomorphism) model $G$ that is $p'$-reduced group (that is, $G$ has no proper normal subgroup of order prime to $p$). Furthermore, the following holds.
\begin{enumerate}[(i)]

\item $\LL \cong \LL_S(G)$ (for any choice of a Sylow $p$-subgroup $S \leq G$).

\item $G \cong \Aut_{\LL}(Q)$ for each $Q \leq S$ which is $\FF$-centric and normal in $\FF$.

\end{enumerate}

\end{prop}

\subsection{Unstable Adams operations}
We end this section with a brief discussion of unstable Adams operations on $p$-local compact groups. Let $\zeta$ be a $p$-adic unit, and let $\g = \SFL$ be a $p$-local compact group. An \hadgesh{unstable Adams operation of degree $\zeta$} on $\g$ is a pair $(\psi,\Psi)$, where $\psi$ is a fusion preserving automorphism on $S$, that restricts to the $\zeta$ power map on the maximal torus $S_0$, and induces the identity on the quotient $S/S_0$. Such an automorphism induces a self  equivalence of $\FF$, which may also be denoted by $\psi$, and the $\Psi$ is an isotypical self equivalence of $\LL$ that covers $\psi$ in the sense that $\psi\rho = \rho\psi$. For a fixed $p$-local compact group $\g$ and a $p$-adic unit of the form $\zeta = 1 + a_kp^k + \cdots$, where $k$ is sufficiently large, it was shown in \cite{JLL} that there exists an unstable Adams operation of degree $\zeta$ on $\g$.

%%%%%%%%%%%%%%%%%%%%%%%%%%%%%%%%%%%%%%%%%%%%
%%%%%%%%%%%%%%%%%%%%%%%%%%%%%%%%%%%%%%%%%%%%
%%%%%% SECTION 2 %%%%%%%%%%%%%%%%%%%%%%%%%%%%%%%
%%%%%%%%%%%%%%%%%%%%%%%%%%%%%%%%%%%%%%%%%%%%
%%%%%%%%%%%%%%%%%%%%%%%%%%%%%%%%%%%%%%%%%%%%
%%%%%%%%%%%%%%%%%%%%%%%%%%%%%%%%%%%%%%%%%%%%

\section{The Robinson Amalgam for Fusion systems over discrete $p$-toral groups}\label{Gog}

In this section we present a slight generalisation of the Robinson amalgam  \cite{Robinson}, realising a given saturated fusion system.  Although our main theorem also holds with respect to the original Robinson amalgam, a certain modification given in \cite{LS}  behaves better with respect to a variety of example.  Most of the discussion here is based on part  of the first author's thesis \cite[Section 2]{Gonza} and on \cite{Gonza3}.  

For our aims it is useful to present these amalgams as  colimits  of  trees of groups 
(in the sense of Serre \cite{Serre}). A (finite) tree $\etz$ with vertices and edges $\etz^0$ and $\etz^1$ respectively,  may be regarded as a category whose object set is the disjoint union $\etz^0\sqcup\etz^1$, and where for each edge object $e \in \etz^1$ one has a unique morphism from $e$ to each of its two vertex objects. We shall denote the category associated to a tree $\etz$ by the same symbol.  A \hadgesh{tree of groups} is a pair $(\etz, \G)$, where  $\etz$ is a tree, $\G\colon\etz\to\grp$ is a functor, and $\grp$ denotes the category of groups. The corresponding amalgam is given by  $\G_\etz = \colim_{\etz}\G$. The following lemma is elementary.

\begin{lmm}\label{lem-fundamental}
Let $(\etz, \G)$ be a tree of groups, and let $\calb: \grp \to \Cat$ be  a functor which associates with a group $\pi$ the one object category with $\pi$ as its group of automorphisms (Thus the geometric realisation $|\calb(-)|$ is the classifying space functor). Then there is a homotopy equivalence $\hocolim_{\etz} |\calb\G|\simeq B\G_\etz$.
\end{lmm}

\begin{proof}
This is a particular case of \cite[Corollary 5.4]{Farjoun}. 
\end{proof}

\begin{Def}\label{defi-fuscontrol}
Let $\FF$ be a saturated fusion system over a discrete $p$-toral group $S$. A \hadgesh{fusion controlling family} for $\FF$ is a finite set   
$$
\pp = \{P_0 = S, P_1, \ldots, P_k\} $$
of fully $\FF$-normalised subgroups, which contains at least one representative for each $\FF$-conjugacy class of $\FF$-centric $\FF$-radical subgroups. We say that a fusion controlling  family $\pp$ is \hadgesh{complete} if it consist of exactly  one  representative for each $N_\FF(S)$-conjugacy class of  $\FF$-centric $\FF$-radical subgroups.
\end{Def}

Notice that  fusion controlling families exist by Proposition \ref{blo3-S3}(c). Also, since $N_\FF(S)$ is a fusion subsystem of $\FF$, every $\FF$-centric $\FF$-radical subgroup of $S$ is also $N_\FF(S)$-centric and $N_\FF(S)$-radical. Since there are only finitely many $N_\FF(S)$-conjugacy classes of subgroups of $S$ which are $N_\FF(S)$-radical, complete fusion controlling families for $\FF$ are finite sets.

\begin{defi}\label{Robinson-Setup}
Let $\g = \ploc$ be a $p$-local compact group. A \hadgesh{Robinson setup for $\g$} consists of a fusion controlling family $\pp = \{P_0 = S, P_1, \ldots, P_k\}$ and for each $P\in \pp$ a pair of groups $(L_P, N_P)$ such that:
\begin{enumerate}[(i)]
\item $N_S = 1$,
\item $L_P = \Aut_{\LL}(P)$, for all $P\in \pp$, and 
\item for $S\neq P\in\pp$, $N_P\le L_P$, $N_P\le L_S$, and $N_S(P)\le N_P$.
\end{enumerate}
For each $S\neq P\in\pp$ we let
\[L_S \Left2{k_P} N_P \Right2{j_P} L_P\]
denote the inclusions.
\end{defi}

Next we associate an amalgam with a given Robinson setup.

\begin{defi}\label{Robgraph}
Let $\g = \ploc$ be a $p$-local compact group, and let $(\pp; \{(L_P, N_P)\}_{P\in\pp})$ be an associated Robinson setup. The \hadgesh{Robinson tree of groups} corresponding to the given setup is defined as the pair $(\etz_{}, \G)$ where
\begin{itemize}
\item
$\etz$ is the tree with one vertex $v_P$ for each object $P\in\pp$, and one edge $e_P$ linking the vertex $v_S$ with $v_P$ for each $S\neq P\in \pp$, considered as a category,  with $\Ob(\etz) = \etz_{}^0 \sqcup \etz_{}^1$, and for each $P\in\pp$, two morphisms $e_P\xto{\iota_P}v_S$ and $e_P\xto{\nu_P}v_P$.
\item
 $\G: \etz_{} \to \grp$ is the functor defined on objects by 
\[\G(v_P)\defin L_P,\quad\text{and}\quad \G(e_P) \defin N_P,\]
 for each $P\in \pp$, and on morphisms by
 \[\G(\iota_P)\defeq k_P\colon N_P\to L_S,\quad\text{and}\quad
\G(\nu_P)\defeq j_P \colon N_P\to L_P.\] 
 \end{itemize}
 The \hadgesh{Robinson amalgam} associated to the given setup is the group
$$
G\defeq \G_{\etz} = \colim_{\etz}(\G)
$$
\end{defi}

The group $G$ in Definition \ref{Robgraph} is clearly isomorphic to the iterated amalgam,
\begin{equation}\label{amalgam-identification}
G'\defeq(\cdots((L_S\ast_{N_1} L_1)\ast_{N_2}L_2)\ast_{N_3}\cdots L_{k-1})\ast_{N_k}L_k,
\end{equation}
where $L_i$ and $N_i$ run over pairs $(L_P, N_P)$ for all $S\neq P\in\pp$.

The following statement is a slight generalisation of \cite[Theorem 1]{Robinson}. The proof, which we leave for the reader, is a minor modifications of Robinson's original proof, and finiteness of the groups involved is not required.

\begin{prop}\label{Geoff2}
Let $H$ and $K$ be (discrete) groups with Sylow $p$-subgroups $S \in Syl_p(H)$ and $P \in Syl_p(K)$. Let  $N \leq K$ be a subgroup such that $P \leq N$, and let $\alpha \colon N \to H$ be a monomorphism such that $\alpha(P) \leq S$. Set $G = H \ast_N K$, where we identify $N$ with its image through $\alpha$ and as a subgroup of $K$. Then $S\in\sylp(G)$, and $\FF_S(G)$ is generated by the subsystems $\FF_S(H)$ and $\FF_P(K)$.
\end{prop}

Proposition \ref{Geoff2} will now be used to prove a slightly restricted generalisation of Robinson's theorem.

\begin{prop}\label{Geoff}
Let $\g = \ploc$ be a $p$-local compact group, and let $(\pp; \{(L_P, N_P)\}_{P\in\pp})$ be an associated Robinson setup. Let $G$ denote the corresponding Robinson amalgam. Then $S\in\sylp(G)$,  and  $\FF_S(G)\cong \FF$.
\end{prop}

\begin{proof}
Fix an isomorphism of $G$ with an iterated amalgam $G'$ as in (\ref{amalgam-identification}), and use Proposition \ref{Geoff2} for an inductive argument. Fix an ordering $\{S=P_0, P_1,\ldots, P_k\}$ on $\pp$, and denote $L_{P_i} = L_i$ and $N_{P_i} = N_i$. For each $0\le i\le k$, $L_i$ is a model for the fusion system $N_{\FF}(P_i)$, and in particular $N_S(P_i) \in \Syl_p(L_i)$ and $\FF_{N_S(P_i)}(L_i)\cong N_{\FF}(P_i)$. Inductively define groups $G'_i$ by setting $G'_0 = L_0$, $G'_{i+1} = G'_i \ast_{N_i} L_i$, for $i = 0, \ldots, k-1$, and  $G' = G'_k$. Then, by repeatedly applying Proposition \ref{Geoff2}, it follows that $S \in \Syl_p(G')$, and $\FF = \FF_S(G')$, as claimed.
\end{proof}

The original Robinson amalgam for a given $p$-local compact groups and a choice of a fusion controlling family $\pp$ is based on the choice of the groups $N_P$ to be $N_S(P)$. In particular $N_S(P)$  is a Sylow $p$-subgroup in $L_P$ for all $P\in \pp$, since  $P$ is assumed fully normalised. The version of the amalgam introduced by Libman and Seeliger in \cite{LS} replaces $N_S(P)$ by 
\[\Aut_\LL(S,P) \defeq \{\varphi\in\Aut_\LL(S)\;|\; \varphi|_P\in\Aut_\LL(P)\}.\]
In \cite{LS} $\Aut_\LL(S,P)$ is denoted $\Aut_\LL(P<S)$. We refer to this type of setup as a \hadgesh{Libman-Seeliger setup}.

We now relate the classifying space of a Robinson amalgam to the $p$-local compact group from which it originates.

\begin{prop}\label{relbgbg} 
Let $\g=\ploc$ be a $p$-local compact group,  let $(\pp ,\{(L_P,N_P)\}_{P\in\pp})$ be an associated  Robinson setup, and let $G$ denote the resulting Robinson amalgam. Let $(\etz, \G)$ denote the associated Robinson graph of groups, and let $\bb\G\colon\etz\to\Cat$ denote the functor which associates to each object  $x\in\etz$ the category $\bb(\G(x))$, and to each morphism the induced functor. Then the following holds.
\begin{enumerate}[(i)]

\item There is a canonical map
$$
\mu \colon \hocolim_{\etz}|\bb\G|\to B\g,
$$
such that the triangle
\begin{equation}\label{triangle}\xymatrix@C=2cm{
&BS\ar[ld]_{\gamma}\ar[rd]^{\delta}\\
\hocolim_{\etz}|\bb\G|  \ar[rr]^{\mu}  && B\g\\
}\end{equation}
commutes strictly, where $\gamma$ and $\delta$ denote the canonical inclusions.

\item Let $\widehat{\LL}$ denote the reduced linking system associated to $\LL$ (Definition\ref{reducedL}). Then $\mu$ induces a functor $\phi: \TT^c_S(G) \to \hLL$ which is the identity on objects, an epimorphism on morphism sets, and makes the following diagram commute:
\[\xymatrix@C=20mm{
\TT_{\FF^c}(S)\ar[r]^{\delta}\ar[d]_{\delta}& \hLL\ar[d]^{\rho}\\
\TT^c_S(G) \ar[r]_{\rho}\ar[ru]^{\phi} & \FF^c}
\]
where $\TT_{\FF^c}(S)\subseteq\TT_S(S)$ is the full subcategory whose objects are the $\FF$-centric subgroups of $S$.

\end{enumerate}
\end{prop}

\begin{proof}
Consider the Grothendieck construction $\mathrm{Gr}(\etz, \bb\G)$ for the functor $\bb\G$ on $\etz$ \cite{Thomason}.
The nerve of $\mathrm{Gr}(\etz, \bb\G)$ is isomorphic (as a simplicial set) to the homotopy colimit of $|\bb\G|$ over $\etz$. We will  construct a functor  $\Phi\colon\mathrm{Gr}(\etz, \bb\G)\to\LL$.  

Since for every object $x\in\etz$, the object set of the category $\bb\G(x)$  is a singleton, the objects of  $\mathrm{Gr}(\etz, \bb\G)$ are in 1--1 correspondence with those of $\etz$, and so we use the same notation  $v_S, e_Q$ and $v_Q$ to denote them. Automorphism groups in $\mathrm{Gr}(\etz, \bb\G)$ are given by $\Aut(v_S)  = L_S$, and for each $Q\in\pp$, $\Aut(v_Q) = L_Q$ and $\Aut(e_Q) = N_Q$, while $\Mor(e_Q,v_S) =  \{\iota_Q\}\times L_S$, and $\Mor(e_Q,v_Q) = \{\nu_P\}\times L_Q$.

The tree $\etz$, considered as a category, is a union of subcategories $\etz_Q$, one for each $S\neq Q\in\pp$,  which take the form 
\[v_S\xlto{\iota_Q} e_Q\xto{\nu_Q} v_Q.\] 
The corresponding description can be given to $\mathrm{Gr}(\etz, \bb\G)$, namely it is a union of full subcategories, one for each $S\neq Q\in\pp$, on the objects $v_S$, $v_Q$ and $e_Q$, and the appropriate morphism sets. Denote the subcategory corresponding to $Q$ by $\mathrm{Gr}_Q$. 
It thus suffices to prove that for each $S\neq Q\in\pp$, there is a functor $\Phi_Q\colon\mathrm{Gr}_Q\to\LL$, such that for any $Q', Q''\in\pp$, $\Phi_{Q'}$, and $\Phi_{Q''}$  coincide on the full subcategory of their domain  on the object $v_S$.

Define $\Phi_Q\colon\mathrm{Gr}_Q\to \LL$  to be the functor which takes $v_Q$ and $e_Q$ to the object $Q\in\LL$, and takes $v_S$ to $S\in\LL$. The automorphism groups $\Aut(v_Q)$, $\Aut(e_Q)$ and $\Aut(v_S)$ are mapped by $\Phi_Q$ to $\Aut_\LL(Q)$, $N_Q\le\Aut_\LL(Q)$, and $\Aut_\LL(S)$ respectively. Any morphism  $(\nu_Q,\varphi)\in\Mor(e_Q,v_Q)$ is taken by $\Phi_Q$ to $\varphi\in L_Q = \Aut_\LL(Q)$. Similarly, for a morphism $(\iota_Q,\psi)\in\Mor(e_Q,v_S)$, $\Phi_Q(\psi) = \psi\circ\iota_{Q,S}$.

It is easy to verify that $\Phi_Q$ thus defined is a functor that  commutes with the obvious inclusions $\bb(S)\to \mathrm{Gr}_Q$ and $\bb(S)\to \LL$. Define $\Phi\colon\mathrm{Gr}(\etz,\bb\G) \to \LL$ by joining the functors $\Phi_Q$ in the obvious way, and let $\mu$ denote the map induced on nerves upon completion of the target. Commutativity of the triangle in Diagram (\ref{triangle}) follows by construction. By Proposition \ref{reconstruct}, $\LL^c_S(\delta)\cong\hLL$, and by Lemma \ref{lem-fundamental}, and \ref{LS-2.1},  $\LL^c_S(\gamma)\cong\TT^c_S(G)$. The last statement follows from naturality of the construction $\LL_S^c(-)$.
\end{proof}

The proof of Proposition \ref{relbgbg} offers another useful observation.
\begin{cor}\label{relbgbt}
In the setup of Proposition \ref{relbgbg}, the map $\mu$ factors through the transporter system $\TT^\HH_S(G)$, where $\HH$ is a family of subgroups of $S$, closed under $\FF$-conjugacy, which contains $\pp$.
\end{cor}
\begin{proof}
The construction of a functor  $\widehat{\Phi}\colon \mathrm{Gr}(\etz,\bb\G) \to \TT^\HH_S(G)$ follows exactly the same argument as in Proposition \ref{relbgbg}, and we let $\widehat{\mu}$ denote the map induced on nerves. Commutativity of the relevant triangle follows at once.
\end{proof}

Notice that, when $\g = \ploc$ is a $p$-local finite group, the reduced linking system $\widehat{\LL}$ equals the linking system $\LL$, and Proposition \ref{relbgbg} implies the existence of a functor $\Phi \colon \TT_S^c(G) \to \LL$. This is not the case in general for $p$-local compact groups, and we have to produce the functor $\Phi$ by other means.

The construction of $\Phi$ in the general case uses the finite approximations of $p$-local finite groups from \cite{Gonza3}. In order to facilitate the reading, we recall here some of the definitions and results from \cite{Gonza3} that  will be used here. 

Let $\g = \ploc$ be a $p$-local compact groups, and let $\widetilde{\LL}$ be the category with object set of $\Ob(\widetilde{\LL}) = \{P \leq S \, | \, P^{\bullet} \in \Ob(\LL)\}$. The morphism sets in $\widetilde{\LL}$ are
\begin{equation}\label{wL}
\Mor_{\widetilde{\LL}}(P, Q) = \{\varphi \in \Mor_{\LL}(P^{\bullet}, Q^{\bullet}) \, | \, \rho(\varphi)|_P \in \Hom_{\FF}(P,Q)\}.
\end{equation}
By \cite[Proposition 1.11]{Gonza3}, $\widetilde{\LL}$ is a transporter system which contains $\LL$ as a full subcategory, and $|\LL| \simeq |\widetilde{\LL}|$. Furthermore, the functor $(-)^{\bullet} \colon \LL \to \LL^{\bullet}$ extends to a functor $(-)^{\bullet} \colon \widetilde{\LL} \to \LL^{\bullet}$.

Let $(\psi,\Psi)$ be a fixed unstable Adams operation on $\g$ of degree $\zeta$ for some $p$-adic unit \[\zeta = 1 + a_np^n + a_{n+1}p^{n+1}+\cdots,\qquad n\geq 1.\]  
For a sufficiently large $n$ such operations exist by \cite{JLL}. For  each $i \geq 0$ let \[(\psi_i,\Psi_i) \defeq (\psi^{p^i}, \Psi^{p^i}).\]  Thus, for each $i$, $(\psi_i, \Psi_i)$ is an unstable Adams operation on $\g$ of degree $\zeta^{p^i}$. By \cite[Proposition 1.14]{Gonza3}, $\Psi$, and hence $\Psi_i$ for all $i$,  extend uniquely to  equivalences of $\widetilde{\LL}$, which we denote $\Psi_i$ by abuse of notation.

For each $i \geq 0$, let $S_i \leq S$ to be the subgroup of fixed points in $S$ under the automorphism $\psi_i$, and note that $S = \bigcup_{i \geq 0} S_i$. Let  $\LL_i \subseteq \widetilde{\LL}$ be the subcategory with object set \[\Ob(\LL_i) = \{H \leq S_i \, | \, H \in \Ob(\widetilde{\LL})\},\] and with morphism sets
\begin{equation}\label{wLi}
\Mor_{\LL_i}(H,K) = \{\varphi \in \Mor_{\widetilde{\LL}}(H,K) \, | \, \Psi_i(\varphi) = \varphi\}.
\end{equation}
Finally, let $\FF_i$ be the fusion system on $S_i$ given by the image in $\FF$ of the structure functor $\rho \colon \widetilde{\LL} \to \FF$  restricted to  $\LL_i$. Notice that there are inclusions
$$
\LL_i \subseteq \LL_{i+1} \subseteq \widetilde{\LL} \qquad \mbox{and} \qquad \FF_i \subseteq \FF_{i+1} \subseteq \FF.
$$

We recall some of the main properties of the triples $(S_i, \FF_i, \LL_i)$ from \cite{Gonza3} in the following. 

\begin{prop}\label{finappr1}
Let $\g = \SFL$ be a $p$-local compact group. Then there exists some $N\in\N$, such that for any $p$-adic unit $\zeta$ that satisfies $p^N | \zeta-1$, there exists an unstable Adams operation $(\psi, \Psi)$ of degree of degree $\zeta$. Furthermore, if we let $(\psi_i, \Psi_i)$ be the $p^i$-th powers of $(\psi,\Psi)$, and  $(S_i, \FF_i, \LL_i)$ be the corresponding fixed point triples, then the following statements hold.
\begin{enumerate}[(i)]
\item For all $i \geq 0$, the category $\FF_i$ is a saturated fusion system over $S_i$ and $\LL_i$ is a quasicentric linking system associated to $\FF_i$
\label{finappr1.1}
\item For each $P, Q \in \Ob(\widetilde{\LL})$ and each $\varphi \in \Mor_{\widetilde{\LL}}(P,Q)$, there exists some $M \in \N$ such that, for all $i \geq M$, $P \cap S_i, Q \cap S_i \in \Ob(\LL_i)$ and $\varphi|_{P \cap S_i} \in \Mor_{\LL_i}(P \cap S_i, Q \cap S_i)$, 
\label{finappr1.2}
\item There is a homotopy equivalence $B\g \simeq (\hocolim |\LL_i|)^{\wedge}_p$, 
\label{finappr1.3}
\end{enumerate}
\end{prop}
\begin{proof}
The first statement of the proposition follows from the main theorem of \cite{JLL}. Parts (\ref{finappr1.1})  and  (\ref{finappr1.2}) follow from 
\cite[Theorem 2.13]{Gonza3} and its proof, by taking the $p$-adic valuation of $\zeta-1$ to be sufficiently large.  Part (\ref{finappr1.3}) is 
\cite[Lemma 2.3]{Gonza3}.
\end{proof}

\begin{lmm}\label{finappr2}

For each $i \geq 0$, let
$$
\hh_i = \{H \leq S_i \, | \, H^{\bullet} \cap S_i = H\},
$$
and let $\LL_{\hh_i} \subseteq \LL_i$ be the full subcategory with object set $\hh_i$. Let  $\LL_i^{\bullet} \subseteq \LL^{\bullet}$ be the image of $\LL_{\hh_i}$ through the functor $(-)^{\bullet} \colon \widetilde{\LL} \to \LL^{\bullet}$. The following properties hold.
\begin{enumerate}[(i)]

\item The functor $(-)^{\bullet}$ induces an isomorphism $\LL_{\hh_i} \cong \LL_i^{\bullet}$, whose inverse on objects is given by taking intersection with $S_i$, and on morphisms by restricting to the corresponding subgroups.

\item The inclusion $\LL_{\hh_i} \subseteq \LL_i$ induces an equivalence $|\LL_{\hh_i}| \simeq |\LL_i|$.

\item For all $i \geq 0$ there are inclusions $\LL_i^{\bullet} \subseteq \LL_{i+1}^{\bullet}$, and $\LL^{\bullet} = \cup_{i \geq 0} \LL_i^{\bullet}$.

\end{enumerate}

\end{lmm}

\begin{proof}
The restriction of the functor $(-)^{\bullet}$ to $\LL_{\hh_i}$ is clearly injective on objects, and hence induces a bijection between the object sets of $\LL_{\hh_i}$ and $\LL_i^{\bullet}$. Since the morphisms in $\LL_{\hh_i}$ are the same as the morphisms between the corresponding objects in $\LL_i$, the functor $(-)^{\bullet}$ is also a bijection on morphism sets. The description of the inverse follows easily. This proves Part (i).

By \cite[Proposition 2.11]{Gonza3}, the collection $\hh_i$ contains all the $\FF_i$-centric $\FF_i$-radical subgroups of $S_i$, and Part (ii) follows by \cite[Theorem B]{BCGLO1}.

It remains to prove Part (iii). For any object $H \in \hh_i$, one has 
$$
H = H^{\bullet} \cap S_i \leq H^{\bullet} \cap S_{i+1} \leq H^{\bullet},$$
where the equality follows by definition of $\hh_i$. By \cite[Lemma 3.2(b)]{BLO3},  $(H^{\bullet} \cap S_{i+1})^{\bullet} = H^{\bullet}$.
In particular, it follows that $H^{\bullet} \cap S_{i+1} \in \Ob(\LL_{\hh_{i+1}})$, and so
$\Ob(\LL_i^{\bullet}) \subseteq \Ob(\LL_{i+1}^{\bullet})$.

Let $\varphi \in \Mor_{\LL_{\hh_i}}(H,K)$. Then $\varphi$ is the restriction of $\varphi^{\bullet} \in \Mor_{\LL_i^{\bullet}}(H^{\bullet}, K^{\bullet})$ to $H$ in the sense that 
\begin{equation}\label{inclusions}
\varepsilon(1)\circ\varphi = \varphi^\bullet\circ\varepsilon(1).
\end{equation} 
By definition of $\LL_i$, $\Psi_i(\varphi) = \varphi$, and $\Psi_i(\varepsilon(1)) = \varepsilon(1)$ since $\Psi_i$ is isotypical. 
Also for any subgroup $P\le S$ such that $\Psi_i(P) = P$, one has $\Psi_i(P^\bullet) = P^\bullet$, since $P^\bullet = P\cdot I$, where $I\le T$ (see \cite[Definition 3.1]{BLO3}). Hence, applying $\Psi_i$ to  Equation (\ref{inclusions}), we get
$$
\varphi^{\bullet} \circ \varepsilon(1) = \varepsilon(1) \circ \varphi = \Psi_i(\varphi^{\bullet}) \circ \varepsilon(1).
$$
Since $\widetilde{\LL}$ is a transporter system, morphisms in $\widetilde{\LL}$ are epimorphisms in the categorical sense by \cite[Proposition A.2 (d)]{BLO6}, and this implies that $\Psi_i(\varphi^{\bullet}) = \varphi^{\bullet}$.
Next, if $g \in H^{\bullet} \cap S_{i+1}$, then
$$
\varphi^{\bullet} \circ \widehat{g} \circ (\varphi^{\bullet})^{-1} = \widehat{\rho(\varphi^\bullet)(g)},
$$
by Axion (C), and $\rho(\varphi^\bullet)(g)\in K^{\bullet} \cap S_{i+1}$, since all three morphisms above are fixed by $\Psi_{i+1}$. Thus $\varphi^{\bullet}$ restricts to a morphism $\varphi' \in \Mor_{\LL_{\hh_{i+1}}}(H^{\bullet} \cap S_{i+1}, K^{\bullet} \cap S_{i+1})$. It follows that $\varphi^{\bullet} \in \Mor(\LL_{i+1}^{\bullet})$. This shows that $\Mor(\LL_i^{\bullet}) \subseteq \Mor(\LL_{i+1}^{\bullet})$ and hence that $\LL_i^{\bullet} \subseteq \LL_{i+1}^{\bullet}$. 

Finally, we have to show that $\LL^{\bullet} = \bigcup_{i \geq 0} \LL_i^{\bullet}$. By \cite[Proposition 2.8 (i)]{Gonza3}, for any $P \in \Ob(\LL^{\bullet})$ there is some $M \in \N$ such that $(P \cap S_i)^{\bullet} = P$ for all $i \geq M$. Hence $P \cap S_i \in \Ob(\LL_{\hh_i})$ for all $i \geq M$, and thus $P \in \Ob(\LL_i^{\bullet})$ for all $i \geq M$. By \cite[Proposition 2.8 (i) and (ii)]{Gonza3}, for any $\varphi \in \Mor_{\LL^{\bullet}}(P,Q)$  there is some $M \in \N$ such that $P, Q \in \Ob(\LL_i^{\bullet})$, and $\Psi_i(\varphi) = \varphi$ for all $i \geq M$. In particular, $\varphi$ restricts to a morphism $\varphi_i \in \Mor_{\LL_{\hh_i}}(P \cap S_i, Q \cap S_i)$, with $\varphi_i^{\bullet} = \varphi$, and thus $\varphi \in \Mor(\LL_i^{\bullet})$ for all $i\geq M$.
\end{proof}

\begin{prop}\label{lift}
Let $\g = \ploc$ be a $p$-local compact group, let $(\pp; \{(L_P, N_P)\}_{P\in\pp})$ be a Robinson setup, and let $G$ be the associated Robinson amalgam. Set $\TT = \TT^c_S(G)$ for short, and let $\TT^{\bullet} \subseteq \TT_S(G)$, and $\LL^\bullet\subseteq \LL$ be the full subcategories with object set $\Ob(\FF^{c})\cap\Ob(\FF^\bullet)$. By abuse of notation let $\phi \colon \TT^{\bullet} \to \widehat{\LL}$ denote the restriction of the functor in Proposition \ref{relbgbg}. Then, there is a functor
$$
\Phi \colon \TT^{\bullet} \longrightarrow \LL^{\bullet},
$$
that is the identity on objects and an epimorphism on morphism sets, and such that $\phi$ factors as the composite
\[\xymatrix{\TT^\bullet \ar[r]^\Phi &\LL^\bullet\ar[r]^\incl &\LL\ar[r]^\proj&\widehat{\LL}.\\}\]
\end{prop}

\begin{proof}
Let $\{(S_i, \FF_i, \LL_i)\}_{i \geq 0}$ be an approximation of $\g$ by $p$-local finite groups associated to an appropriate choice of an unstable Adams operation $(\psi, \Psi)$ of degree $\zeta$ on $\g$, as in Proposition \ref{finappr1}. For each $i \geq 0$, let $\LL_{\hh_i} \subseteq \LL_i$ and $\LL_i^{\bullet} \subseteq \LL^{\bullet}$ be the subcategories defined in Lemma \ref{finappr2}. For each $P \in \pp$ and each $i \geq 0$, set $P_i = P \cap S_i$. Throughout this proof we will denote the maximal torus of a discrete $p$-toral group $H$ by $T_H$, and if  $H = S$ the maximal torus will be denoted by $T$. 

We start with a few observations.
\begin{itemize}

\item[(a1)] We may assume $P_i^{\bullet} = P$, for each $P\in\pp$: By \cite[Proposition 2.8 (i)]{Gonza3}, for each $P \in \pp$ there is some $M_P \in \N$ such that $P_i^{\bullet} = P$ for all $i \geq M_P$.  Since the set $\pp$ is finite, our assumption can be satisfied with an appropriate choice of the Adams operation generating the approximation. 

\item[(a2)] $P_i \in \Ob(\LL_i)$: By (a1), $P_i^{\bullet} \in \pp$, and $\pp$ is contained in the object set of the category $\widetilde{\LL}$ defined in (\ref{wL}). Thus, by definition of $\LL_i$ the subgroup $P_i$ is an object of $\LL_i$.

\item[(a3)] Since $S_i$ is the subgroup of fixed points by the automorphism $\psi$, it follows that $T_i = T\cap S_i$ is the subgroup of $T$ consisting of all elements of exponent dividing $p^{np^i}$, where $p^n$ is the highest power of $p$ dividing $\zeta-1$. More generally, for each $H \leq S$, the intersection $T_H \cap S_i$ is the subgroup of $T_H$ of all elements of exponent dividing $p^{np^i}$.

\end{itemize}
As a consequence of (a1), (a2), and the definition of $\LL_i^{\bullet}$, we may assume that
\begin{equation}\label{pp in Li}
\pp = \{P_0 = S, P_1, \ldots, P_k\} \subseteq \Ob(\LL_i^{\bullet})
\end{equation}
for all $i$. Finally, for each $P \in \pp$, denote $L(P) = L_P = \Aut_{\LL}(P)$ and $N(P) = N_P$, as in the Robinson setup in the statement. Thus,
$$
G = L(P_0) \ast_{N(P_1)} L(P_1) \ast_{N(P_2)} \ldots \ast_{N(P_k)} L(P_k).
$$

%%%%%%%%%%%%%%%
\noindent\textbf{Step 1.} (Approximation of the group $G$) For each $P \in \pp$ and each $i \geq 0$, set $L_i(P) = \Aut_{\LL_i^{\bullet}}(P)$ and $N_i(P) = N(P) \cap L_i(P)$. Thus
$$
L(P) = \bigcup_{i \geq 0} L_i(P) \qquad \mbox{and} \qquad N(P) = \bigcup_{i \geq 0} N_i(P).
$$
For $i \geq 0$, define  
$$
G_i \defeq L_i(P_0) \ast_{N_i(P_1)} L_i(P_1) \ast_{N_i(P_2)} \ldots \ast_{N_i(P_k)} L_i(P_k).
$$
Notice that we do not require $N_{S_i}(P) \in \Syl_p(L_i(P))$ for any $P \in \pp$ or $i \geq 0$, and thus we do not require either that $S_i \in \Syl_p(G_i)$. However, there are distinguished inclusion homomorphisms
$$
S_i \longrightarrow L_i(P_0) \longrightarrow G_i,
$$
and we identify $S_i$ with its image in $G_i$.

The relationship among the groups $G_i$ and $G$ is best described by thinking about them as colimits. Let $\etz = \etz_\pp$, and let $(\etz,\G)$ be the Robinson tree associated to the family $\pp$ (Definition \ref{Robgraph}) an. For each $i\geq 0$ one has analogous trees of groups $(\etz, \G_i)$, where $\G_i$ associates $L_i(P_j)$ with $v_{P_j}$, $N_i(P_j)$ with $e_{P_j}$ and the respective inclusions with the morphisms in $\etz$. Then
\[G = G_\pp \cong \colim_{\etz}\G,\quad\text{and}\quad G_i =  \colim_{\etz}\G_i.\]
There are obvious inclusion natural transformations $\G_i\to \G_{i+1}$ for each $i\geq 0$, and hence for each $i$ on has group homomorphisms 
\[\omega_i\colon G_i\to G_{i+1},\quad\text{and}\quad \sigma_i \colon G_i \to G.\] By naturality of the colimit construction for each $i\geq 1$, $\sigma_{i+1}\circ\omega_i = \sigma_i$. By commutation of colimits it also follows at once that $G = \colim_i G_i$. Notice also that 
\[\hocolim_\etz|\calb\G_i|\simeq BG_i,\quad\text{and}\quad \hocolim_\etz|\calb\G| \simeq BG,\]
as in Lemma \ref{lem-fundamental}. 

%%%%%%%%%%%%%%%
\noindent\textbf{Step 2.} (The categories $\TT_i^{\bullet}$) For each $i$, let $\TT_i^{\bullet}$ be the category with object set  $\Ob(\TT_i^{\bullet}) = \Ob(\LL_i^{\bullet})$, and  with morphism sets 
$$
\Mor_{\TT_i^{\bullet}}(H,K) = \sigma_i^{-1} \{N_G(H_i, K_i) \cap N_G(H,K) \}
$$
for all $H, K \in \Ob(\TT_i^{\bullet})$. Composition of morphisms is given by multiplication in $G_i$. To check that composition is well defined and associative, let $H, K, N \in \Ob(\TT_i^{\bullet})$, and let $g \in \Mor_{\TT_i^{\bullet}}(H,K)$ and $h \in \Mor_{\TT_i^{\bullet}}(K,N)$. By definition, $\sigma_i(g) \in N_G(H_i, K_i) \cap N_G(H,K)$ and $\sigma_i(h) \in N_G(K_i, N_i) \cap N_G(K,N)$. Thus
$$
\sigma_i(h \cdot g) \in N_G(H_i, N_i) \cap N_G(H,N),
$$
since $\sigma_i$ is a homomorphism. Associativity follows from associativity of multiplication in $G_i$.

%%%%%%%%%%%%%%%
\noindent\textbf{Step 3.} (Functors $\TT_i^{\bullet} \to \TT_{i+1}^\bullet$) 
Let $H, K \in \Ob(\TT_i^{\bullet})$. We claim that for all $i\leq j$,
$$
N_G(H_i, K_i) \cap N_G(H, K) \subseteq N_G(H_j, K_j) \cap N_G(H, K).
$$
Note that $H = H_i \cdot T_H$ and $K = K_i \cdot T_K$, where $T_H$ and $T_K$ are the corresponding maximal tori. Since $H_i \leq H_j$ and $H_i^{\bullet} = H$ (and similarly for $K_i, K_j$), it follows that
$$
H_j = H_i \cdot (T_H\cap S_j) \qquad \mbox{and} \qquad K_j = K_i \cdot (T_K\cap S_j).
$$
Let $g \in N_G(H_i, K_i) \cap N_G(H, K)$. By definition, $g \cdot H_i \cdot g^{-1} \leq K_i$ and $g \cdot H \cdot g^{-1} \leq K$. Thus, in particular, $g \cdot T_H \cdot g^{-1} \leq T_K$. As a consequence of property (a3), one has $g(T_H\cap S_j)g^{-1}\le T_K\cap S_j$. Hence
$$
g \cdot H_j \cdot g^{-1} = g \cdot H_i \cdot (T_H\cap S_j) \cdot g^{-1} \leq K_i \cdot (T_K\cap S_j) = K_j,
$$
and so  $g \in N_G(H_j, K_j) \cap N_G(H, K)$.

Define functors $\widetilde{\omega}_i \colon \TT_i^{\bullet} \to \TT_{i+1}^{\bullet}$ for all $i \geq 1$. For each $H \in \Ob(\TT_i^{\bullet})$ and each $g \in \Mor(\TT_i^{\bullet})$, set $\widetilde{\omega}_i(H) = H$ and $\widetilde{\omega}_i(g) = \omega_i(g)$. Note that if $H_i^{\bullet} = H$ then $H_{i+1}^{\bullet} = H$, so $\widetilde{\omega}_i$ is well defined on objects. To check that it is well defined on morphisms, we have to show that
$$
\widetilde{\omega}_i(g) \in \Mor_{\TT_{i+1}^{\bullet}}(H,K) \defin \sigma_{i+1}^{-1}\{N_G(H_{i+1}, K_{i+1}) \cap N_G(H,K)\}.
$$
By Step 2, $\sigma_i(g) \in N_G(H_i, K_i) \cap N_G(H,K)$, and the claim follows from the previous paragraph together with the relation $\sigma_{i+1}\circ\omega_i = \sigma_i$ shown in Step 1.

%%%%%%%%%%%%%%%
\noindent\textbf{Step 4.} (Functors $\Phi_i \colon \TT_i^{\bullet} \to \LL_i^{\bullet}$) By Lemma \ref{finappr2} (i) and (ii), there are homotopy equivalences $|\LL_i^{\bullet}| \simeq |\LL_{\hh_i}| \simeq |\LL_i|$. The same argument as in the proof of Proposition \ref{relbgbg} gives maps
$$
\Phi_i \colon BG_i \simeq \hocolim_\etz|\calb\G_i| \to |\LL_i^{\bullet}|^{\wedge}_p,
$$
such that the  triangles
$$
\xymatrix{
 & BS_i \ar[rd]^{\delta_i} \ar[ld]_{\gamma_i} & \\
BG_i \ar[rr]^{\Phi_i} & & |\LL_i^{\bullet}|^{\wedge}_p
}
$$
commute, for each $i$. In particular, the map $\Phi_i$ induces a functor
$$
\LL_{S_i}^{\hh}(\gamma_i) \cong \TT^{\hh}_{S_i}(G_i) \Right6{\phi_i} \LL_{S_i}^{\hh}(\delta_i),
$$
as in (\ref{induced functors}) of Section \ref{ConstructFL}.

The space $|\LL_i|^{\wedge}_p$ is the classifying space of a $p$-local finite group by \cite[Remark 2.4]{Gonza3}, and it follows that
$$
\LL_i^{\hh}(\delta_i) \cong \LL_{\hh_i} \cong \LL_i^{\bullet}.
$$
Let $\TT_{\hh_i} \subseteq \TT_{S_i}^{\hh}(G_i)$ be the subcategory with object set $\hh$ and with morphism sets
$$
\Mor_{\TT_{\hh_i}}(H\cap S_i, K\cap S_i) = \sigma_i^{-1}\{N_G(H\cap S_i, K\cap S_i) \cap N_G(H, K)\},
$$
Then there is a 1--1 correspondence between the objects of $\TT_i^{\bullet}$ and those of  $\TT_{\hh_i}$, under which the corresponding morphism sets coincide, so  $\TT_i^{\bullet}\cong\TT_{\hh_i}$.  By abuse of notation, let $\Phi_i$ denote the composite
$$
\TT_i^{\bullet} \cong \TT_i^{\hh} \subseteq \TT_{S_i}^{\hh}(G_i) \Right4{\phi_i} \LL_i^{\hh}(\delta_i) \cong \LL_{\hh_i} \Right4{(-)^{\bullet}} \LL_i^{\bullet}.
$$
By construction, the functor $\Phi_i$ is the identity on objects. Furthermore, the following square is commutative for each $i$
\begin{equation}\label{TL-compatible}
\xymatrix{
\TT_{i}^{\bullet} \ar[r]^{\Phi_{i}} \ar[d]_{\widetilde{\omega}_i}& \LL_{i}^{\bullet} \ar[d]^{\incl}\\
\TT_{i+1}^{\bullet}  \ar[r]^{\Phi_{i+1}} & \LL_{i+1}^{\bullet} 
}
\end{equation}
In addition, we claim that  for any $P \in \pp$, there is an inclusion map $\Aut_{\LL_i^{\bullet}}(P) \xto{\incl} \Aut_{\TT_i^{\bullet}}(P)$, such that  the composition 
\begin{equation}\label{retracting-autos}
\Aut_{\LL_i^{\bullet}}(P) \Right3{\incl} \Aut_{\TT_i^{\bullet}}(P) \Right3{\Phi_i} \Aut_{\LL_i^{\bullet}}(P)
\end{equation}
is an automorphism. To see this, fix some $P\in\pp$. By definition, $L_i(P) \leq L(P) \leq N_G(P)$, and  $L_i(P) = \Aut_{\LL_i^{\bullet}}(P)$. Thus  the isomorphism $\LL_i^{\bullet} \cong \LL_{\hh_i}$ of Lemma \ref{finappr2} (i) implies, in particular,  that the elements of $\Aut_{\LL_i^{\bullet}}(P)$ restrict to automorphisms in $\Aut_{\LL_{\hh_i}}(P \cap S_i)$. It follows that $L_i(P) \leq N_G(P_i)$, and thus that
$$
L_i(P)\le N_G(P_i)\cap N_G(P).
$$
By definition $\Aut_{\TT^\bullet_i}(P)= \sigma_i^{-1}(N_G(P_i)\cap N_G(P))$, and $\sigma_i|_{L_i(P)}$ is the identity. Hence we obtain the inclusion \[\Aut_{\LL_i^{\bullet}}(P) \xto{\incl} \Aut_{\TT_i^{\bullet}}(P).\]
The claim that $\Phi_i\circ\incl$ is an automorphism follows at once.

%%%%%%%%%%%%%%%
\noindent\textbf{Step 5.} (The Functor $\Phi \colon \TT^{\bullet} \to \LL^{\bullet}$) Since $G \cong \colim_i G_i$, the following properties hold for each $H, K \in \Ob(\TT^{\bullet})$ and each $g \in N_G(H,K)$.
\begin{enumerate}[(a)]

\item There exists some $M_g$ such that $\sigma_i^{-1}(g) \neq \emptyset$ for all $i \geq M_g$.

\item Let $\sigma_{i,m} \colon G_i \to G_m$ denote the natural map. If $g' \in \sigma_i^{-1}(g)$ and $g'' \in \sigma_j^{-1}(g)$ (for some $i,j \geq M_g$), then there exists some $M \geq i, j$ such that, for all $m \geq M$,
$$
\sigma_{i,m}(g') = \sigma_{j,m}(g'') \in \sigma^{-1}_m(g).
$$
\end{enumerate}
Define $\Phi \colon \TT^{\bullet} \Right2{} \LL^{\bullet}$ as the identity on objects. On morphisms, $\Phi(g) = (\incl \circ \Phi_i)(g')$ for some $i \geq M_g$ and some $g' \in \sigma_i^{-1}(g)$. Properties (a) and (b) above, together with Diagram (\ref{TL-compatible}), imply that $\Phi$ is well-defined. Furthermore, since $\pp$ contains representatives of all the $\FF$-conjugacy classes in $\LL^{\bullet}$, it follows by Statement (a) in Step 4 that $\Phi$ is surjective on morphism sets. Finally, the factorisation \[\phi = \proj\circ\incl\circ\Phi\] follows by construction.
\end{proof}

We end this section with one more observation. For a $p$-local compact group $\g = \ploc$, define the \textit{centre of $\g$}, $Z(\g)$, as
$$
Z(\g) \defin \invlim{\mathcal{O}(\FF^c)} \ZZ_{\FF},
$$
where $\ZZ_{\FF} \colon \mathcal{O}(\FF^c)^{\op} \longrightarrow \Ab$ is the functor that sends each $P$ to its centre.

\begin{prop}\label{ZGZF}
Let $\g = \ploc$ be a $p$-local compact group, let $(\pp; \{(L_P, N_P)\}_{P\in\pp})$ be a Robinson setup, and let $G$ be the associated Robinson amalgam. Then $Z(G) = Z(\g)$.
\end{prop}

\begin{proof}
Write $G = ((L_S \ast_{N_1} L_1) \ast_{N_2} \ldots ) \ast_{N_k} L_k$, let $G_0 = L_S$, and for $j = 1, \ldots, k$ let $G_j \leq G$  be the amalgam over the first $j+1$ factors. Similarly, let for $j = 0, \ldots, k$, let $\LL_j \subseteq \LL$ and $\mathcal{O}_j\subseteq \mathcal{O}(\FF^c)$ be the full subcategories with object sets $\{S=P_0, \conj{P_1}{\FF}, \ldots, \conj{P_j}{\FF}\}$. For each $j\geq 0$ define $Z_j \defin \invlim{\mathcal{O}_j} (\ZZ_\FF)$. We will show inductively that $Z(G_j) = Z_j$. For j=0 this is clear.

By construction $Z_j\leq Z(G_j)$, so it remain to prove the opposite inclusion. For $j \geq 1$ write
$$
G_j = G_{j-1} \ast_{N_j} L_j,
$$
and let  $z \in Z(G_j)$. There are two possibilities: either $z\in G_{j-1}$ or $z\notin G_{j-1}$. We deal with each of them separately. 
\begin{enumerate}[(1)]
\item If $z \in G_{j-1}$, then $z \in Z_{j-1}\le Z(S)$ by induction hypothesis, and $Z(S)\le N_j$ since $P_j$ is $\FF$-centric. Notice also that $z$ commutes with every element of $L_j$ by assumption.

\item If $z \notin G_{j-1}$, apply Robinson's lemma \cite[Lemma 1]{Robinson} with $X = A = G_{j-1}$, $B = L_j$ and $C = N_j$. Write $z=a_0b_1a_1\ldots b_sa_sb_\infty$, where $ a_0\in G_{j-1}, b_\infty\in L_j$, and for $1\le i\le s$, $a_i\in G_{j-1}\setminus N_j$, and $b_j\in L_j\setminus N_j$.  The claim of Robinson's  lemma is  that 
\[\langle G_{j-1}^{a_0}, G_{j-1}^{a_0b_1},\ldots, G_{j-1}^{a_0b_1\cdots a_s}\rangle \le N_j,\]
which implies that $G_{j-1} = N_j$, and so $G_j = L_j$. Thus $z\in Z(L_j)\le Z(S)$, since $S$ is centric in $L_j$. Notice that in this case, $G_j\geq L_i$ for all $i\le j$, so $z\in Z(L_i)$ for all $i\le j$.
\end{enumerate}
In either case $z$ commutes with $L_i$ for all $i\le j$, and since the elements of these subgroups generate all morphisms in $\mathcal{O}_j$, it follows that $z\in Z_j$, as claimed.
\end{proof}

 %%%%%%%%%%%%%%%%%%%%%%%
 %%%%%%%%%%%%%%%%%%%%%%%
 %%%% SECTION 3 %%%%%%%%%%%%%
 %%%%%%%%%%%%%%%%%%%%%%%
 %%%%%%%%%%%%%%%%%%%%%%%
 
 \section{Automorphisms of $p$-local compact groups  Robinson Amalgams}

Throughout this section let $\g = \ploc$ be a fixed $p$-local compact group.  Fix a Robinson setup $(\pp; \{(L_P, N_P)\}_{P\in\pp})$, and let $G$ denote the corresponding Robinson amalgam. The following automorphism groups will appear throughout this section:
\begin{itemize}
\item $\atyp^I(\LL) \leq \atyp(\LL)$, the subgroup of inclusion preserving isotypical self equivalences of $\LL$ (Definition \ref{aut^I}), and

\item $\Aut(G, S) \leq \Aut(G)$, the subgroup of automorphisms of $G$ that  restrict to an automorphism of $S$.
\end{itemize}

We aim to construct a group homomorphism 
$$\Omega \colon \Aut(G, S) \to \atyp(\LL).$$
The following two lemmas are preliminary.

\begin{lmm}\label{out-pi}
Let $H$ be a discrete group which admits a Sylow  $p$-subgroup (See Definition \ref{Sylow}) $S\le H$. Then there is a an exact sequence
\[\xymatrix{ 1 \ar[r] & Z(H)  \ar[r] & N_{H}(S)  \ar[r] & \Aut(H, S) \ar[r] & \Out(H) \ar[r] & 1
}
\]
\end{lmm}
\begin{proof} Let $\alpha$ be any automorphism of $H$. Since $S$ is a Sylow subgroup of $H$, it is unique up to conjugacy. Hence there is some $h\in H$ such that $h\cdot\alpha(S) \cdot h^{-1}  = c_h\circ\alpha(S) = S$. This proves exactness at $\Out(H)$. Exactness at $\Aut(H, S)$ follows from the observation that the image of $N_{H}(S)$ in $\Aut(H, S)$ is exactly $\Aut(H, S)\cap\Inn(H)$. The rest is immediate.
\end{proof}

\begin{lmm}\label{cbullet}
Let $H$ be a discrete group which admits a Sylow  $p$-subgroup  $S\le H$, and let $\alpha\in\Aut(H, S)$ be an automorphism. Then $\alpha$ induces an automorphism on $\TT_S(H)$, also denoted by $\alpha$. In particular, the restriction $\alpha|_S$ preserves $H$-fusion among subgroups of $S$. Furthermore, if $\FF = \FF_S(H)$, then the object sets $\Ob(\FF^{c r})$, $\Ob(\FF^c)$ and $\Ob(\FF^{\bullet})$ are preserved by $\alpha$. 
\end{lmm}
\begin{proof}
Fix an element $\alpha\in\Aut(H, S)$, let $P, Q\le S$ be subgroups, and let $h\in N_H(P,Q)$ be any element. Then $\alpha(h)\in N_H(\alpha(P),\alpha(Q))$, and the first claim follows at once. In particular, $\alpha \circ c_h\circ\alpha^{-1} = c_{\alpha(h)}$, so $\alpha|_S$ preserves $H$-fusion among subgroups of $S$, as claimed.
Finally, if $\FF = \FF_S(H)$, then for any $P, Q\le S$, $\Hom_\FF(P,Q) = \Hom_{H}(P,Q)$. 
The last statement follows. 
\end{proof}

\begin{prop}\label{verybonito2}
Let $\g = \ploc$ be a $p$-local compact group, let $(\pp; \{(L_P, N_P)\}_{P\in\pp})$ be a Robinson setup, and let $G$ be the associated Robinson amalgam.  There is a group homomorphism $\Omega\colon \Aut(G, S) \longrightarrow \atyp^I(\LL)$ such that the following square 
\begin{equation}\label{vb2-diag}
\xymatrix{
BG \ar[r]^{\mu} \ar[d]_{B\alpha} & B\g \ar[d]^{|\Omega(\alpha)|\pcom} \\
BG \ar[r]_{\mu} & B\g \\
}
\end{equation}
commutes up to homotopy, for each $\alpha \in \Aut(G, S)$, where $\mu\colon BG\to B\g$ denotes the map  constructed in Proposition \ref{relbgbg}.
\end{prop}

\begin{proof}
Let $\TT = \TT_S^c(G)$, and let  $\alpha \in \Aut(G, S)$ be any element. By Lemma \ref{cbullet}, $\alpha$ induces an isotypical  self equivalence of $\TT^{\bullet}\subseteq\TT$, which we also denote by $\alpha$.

Let $\Phi \colon \TT^{\bullet} \to \LL^{\bullet}$ be the functor in Proposition \ref{lift}. Let $P \in \Ob(\LL^{\bullet}) = \Ob(\TT^\bullet)$, and let
$$
K(P) \defin \Ker(\Phi_P\colon \Aut_{\TT^{\bullet}}(P) \to \Aut_{\LL^{\bullet}}(P)).
$$
By construction, and since both $\Aut_{\TT^{\bullet}}(P)$ and $\Aut_{\LL^{\bullet}}(P)$ surject onto $\Aut_{\FF}(P)$, it follows that there is a short exact sequence,
$$
1 \to K(P) \longrightarrow C_{G}(P) \longrightarrow Z(P) \to 1.
$$

Since $\alpha$ preserves $S$, $\alpha(P)\le S$, and so the restrictions of $\alpha$ to   $C_{G}(P)$ and  $Z(P)$ define isomorphisms to  $C_{G}(\alpha(P))$ and $Z(\alpha(P))$ respectively. Thus the restriction of $\alpha$ to $K(P)$ maps it to $K(\alpha(P))$ and the diagram $$
\xymatrix{
1 \ar[r] & K(P) \ar[r] \ar[d] & C_{G}(P) \ar[r] \ar[d]_{\alpha} & Z(P) \ar[r] \ar[d]^{\alpha} & 1 \\
1 \ar[r] & K(\alpha(P)) \ar[r] & C_{G}(\alpha(P)) \ar[r] & Z(\alpha(P)) \ar[r] & 1 \\
}
$$
Commutes. Furthermore, for each $P, Q\in\Ob(\LL^\bullet)$, $\alpha$ induces a bijection
\[\alpha_{P,Q}\colon N_{G}(P,Q)\to N_{G}(\alpha(P),\alpha(Q)).\]
Thus $\alpha$  induces a bijection
\[\alpha_{P,Q}\colon \Mor_{\LL^{\bullet}}(P,Q) = N_{G}(P, Q)/K(P) \to N_{G}(\alpha(P),\alpha(Q))/K(\alpha(P)) = \Mor_{\LL^{\bullet}}(\alpha(P),\alpha(Q)).
\]
Since $\alpha(K(P)) = K(\alpha(P))$ for all $P\in \LL^\bullet$, and since  the restriction of $\alpha$ to $S$ preserves $S$-transporter sets, $\alpha$ induces an isotypical self equivalence $\alpha^\bullet$ of $\LL^\bullet$. By Proposition \ref{autext}, there is a unique extension of $\alpha^\bullet$ to an isotypical self equivalence $\Omega(\alpha)$ of $\LL$, which in particular preserves $\TT_S(S)$, and hence is an inclusion preserving self equivalence.

The construction of $\alpha^\bullet$ (and thus also of $\Omega(\alpha)$) is clearly compatible with compositions of automorphisms in $\Aut(G, S)$, which shows that 
\[\Omega\colon  \Aut(G, S)\to \Aut^I_\typ(\LL)\]
is a group homomorphism.

To finish the proof it suffices to show that the diagram
$$
\xymatrix{
BG \ar[r]^{\widehat{\mu}} \ar[d]_{B\alpha} & |\TT^{\bullet}| \ar[r]^{|\Phi|} \ar[d]_{|\alpha|} & |\LL^{\bullet}| \ar[d]^{|\Omega(\alpha)|} \\
BG \ar[r]_{\widehat{\mu}} & |\TT^{\bullet}| \ar[r]_{|\Phi|} & |\LL^{\bullet}|
}
$$
commutes up to homotopy, where $\widehat{\mu}$ is the map constructed in Corollary \ref{relbgbt}.  But the left square commutes by naturally of the construction of $\widehat{\mu}$ (see Proposition \ref{relbgbg} and Corollary \ref{relbgbt}), and  the right square strictly commutes by construction. 
\end{proof}

\begin{prop}\label{morph2}
The homomorphism $\Omega$ induces a homomorphism $\omega\colon \Out(G) \to \otyp(\LL)$.
\end{prop}
\begin{proof}
By Lemma \ref{AOV-1.14}(c), $\otyp(\LL) = \Aut_\typ^I(\LL)/\Aut_\LL(S)$. If $g \in N_{G}(S)$ is any element, and $c_g\in \Aut(G, S)$ is the corresponding inner automorphism, then by definition $\Omega(c_g)\in\Aut_\typ^I(\LL)$ is the equivalence induced by conjugation by $[g] \in N_{G}(S)/K(S) = \Aut_{\LL}(S)$. Hence $\Omega$ induces the homomorphism $\omega$ as claimed.
\end{proof}

%%%%%%%%%%%

Next, we show that if $\pp$ is a complete fusion controlling family, then every inclusion preserving isotypical self equivalence of $\LL$ gives rise to  an automorphism of $G=G_\pp$. While this construction does not produce a homomorphism $\atyp^I(\LL) \to \Aut(G)$, it does induce a right inverse for $\omega$.

\begin{lmm}\label{permutation}
Let $\g = \ploc$ be a $p$-local compact group, and let $\pp = \{S=P_0, \ldots, P_k\}$ be a complete fusion controlling family for $\FF$. Then, for each $\Psi \in \Aut_{\typ}^{I}(\LL)$, the collection $\Psi(\pp) = \{\Psi(P_0), \ldots, \Psi(P_k)\}$ is a complete fusion controlling family. Furthermore, there is a group homomorphism $\upsilon\colon\Out_\typ(\LL)\to\Sigma_{k}$, defined by the relation
\[P_{\upsilon(\Psi)(i)}\in \Psi(P_i)^{N_\FF(S)} \qquad 1 \leq i \leq k.\]
\end{lmm}

\begin{proof}
Any two complete fusion controlling families for $\FF$ must have the same (finite) cardinality. Since $\Psi$ is fusion preserving,  it carries $\FF$-centric, $\FF$-radical subgroups to $\FF$-centric, $\FF$-radical subgroups. Hence it is enough to show that if $P, Q\in \pp$, and $\Psi(P)$ is $N_\FF(S)$-conjugate to $\Psi(Q)$, then $P$ is $N_\FF(S)$-conjugate to $Q$.  Let $f\colon\Psi(P)\to\Psi(Q)$ be an isomorphism in $N_\FF(S)$. Since $\Psi$ is $\FF$-fusion preserving, it is also $N_\FF(S)$-fusion preserving, and hence the same applies to $\Psi^{-1}$. This $\Psi^{-1}\circ f\circ\Psi\colon P\to Q$ is an isomorphism in $N_\FF(S)$, so $P$ is $N_\FF(S)$-conjugate to $Q$, contradicting completeness of $\pp$.

It follows that for each $1\le i\le k$, $\Psi(P_i)$ is $N_\FF(S)$-conjugate to a unique $P_j\in\pp$, for some $j\geq 1$. In other words, $\Psi$ permutes the set of indices $\{1, 2,\ldots, k\}$ associated to the $N_\FF(S)$-conjugacy classes of the members of $\pp$, and thus defines an element of $\Sigma_k$. This correspondence is clearly multiplicative, and hence defines a homomorphism 
\[\upsilon'\colon \Aut_\typ^I(\LL)\to \Sigma_k.\]
If $\Psi$ is induced by conjugation by an element of $\Aut_\LL(S)$, then it clearly acts trivially on $N_\FF(S)$ conjugacy classes of subgroups of $S$.  Hence $\upsilon'$ induces a homomorphism
\[\upsilon\colon\Out_\typ(\LL)\to \Sigma_k\]
as claimed.
\end{proof}

\begin{lmm}\label{autext2}
Let $\pp = \{P_0=S, P_1, \ldots, P_k\}$ be a complete fusion controlling family for $\FF$, and let $\LL_{\pp} \subseteq \LL^{\bullet}$ be the full subcategory with object set $\pp$. If $\Psi, \Psi' \in \Aut_{\typ}^{I}(\LL)$ are such that $\Psi|_{\LL_{\pp}} = \Psi'|_{\LL_{\pp}}$, then $\Psi = \Psi'$.
\end{lmm}

\begin{proof}
Since $\Aut_{\typ}^{I}(\LL)$ is a group, it suffices to show that if $\Psi\in\Aut_{\typ}^{I}(\LL)$ restricts to the identity transformation on $\LL_{\pp}$, then it is the identity on $\LL$. 

Fix an automorphism $\Psi\in\Aut_\typ^I(\LL)$ such that $\Psi|_{\LL_\pp}$ is the identity functor. Since $P_0=S$, $\Psi$ restricted to $\Aut_{\LL_\pp}(S) = \Aut_\LL(S)$ is the identity homomorphism. In particular $\Psi|_S = 1_S$, and hence $\Psi(P) = P$ for each $P\in\LL$. Let $P\le S$ be any $\FF$-centric, $\FF$-radical subgroup, and let $P_i\in\pp$ be the unique representative of the $N_\FF(S)$-conjugacy class of $P$. Let $[\varphi]\colon P_i\to P$ be an isomorphism in $N_\FF(S)$. By definition of the normaliser fusion system, there is an automorphism $[\widetilde{\varphi}]\in\Aut_{\FF}(S)$ such that $[\widetilde{\varphi}]|_{P_i} = \varphi$. Let $\varphi\in\Iso_\LL(P_i,P)$ and $\widetilde{\varphi}\in\Aut_\LL(S)$ be arbitrary lifts of $\varphi$ and $\widetilde{\varphi}$ respectively. Let $\iota_P$ and $\iota_{P_i}$ denote the inclusions in $\LL$ of $P$ and $P_i$ in $S$. Then there exist some $z\in Z(P_i)$ such that 
\[\iota_P\circ\varphi = \widetilde{\varphi}\circ\iota_{P_i}\circ \widehat{z}.\]
Applying $\Psi$, which respects inclusions in $\LL$, and which restricts to the identity on $\Aut_\LL(S)$, we obtain
\[\iota_P\circ\Psi(\varphi) = \widetilde{\varphi}\circ\iota_{P_i}\circ \widehat{z} = \iota_P\circ\varphi.\]
Since $\iota_P$ is a monomorphism $\Psi_(\varphi) = \varphi$. Since both $[\varphi]$ and its lift $\varphi$ were chose arbitrarily, this shows that $\Psi$ restricts to the identity on the set $\Iso_\LL(P_i, P)$. Furthermore, any $\varphi\in\Iso_\LL(P_i, P)$ induces an isomorphism 
\[c_\varphi\colon \Aut_\LL(P_i)\to \Aut_\LL(P)\]
by taking an automorphism $\gamma$ of $P_i$ to $\varphi\circ\gamma\circ\varphi^{-1}$. Since $\Psi$ restricts to the identity map on $\Aut_\LL(P_i)$, it follows that its restriction to $\Aut_\LL(P)$ is the identity map. Hence $\Psi$ restricts to the identity map on all automorphism groups of  $\FF$-centric, $\FF$-radical subgroups $P\le S$. As an immediate consequence of Alperin's fusion theorem \cite[Theorem 3.6]{BLO3}, it now follows that $\Psi$ is the identity functor on $\LL$.
\end{proof}

\begin{prop}\label{morph1}
Let $\g = \SFL$ be a $p$-local compact group, let  $(\pp; \{(L_P, N_P)\}_{P\in\pp})$ be a Robinson setup, where $\pp$ is a complete fusion controlling family for $\FF$, and let $G$ be be the corresponding Robinson amalgam. There is a group homomorphism
$$
\gamma \colon \otyp(\LL) \Right3{} \Out(G)
$$
such that $\omega \circ \gamma = \Id$ on $\otyp(\LL)$, where $\omega\colon \Out(G)\to\Out_{\typ}(\LL)$ is the homomorphism in Proposition \ref{morph2}.  In particular, $\omega \colon \Out(G) \to \otyp(\LL)$ is surjective.
\end{prop}

\begin{proof}
Let $\pp = \{P_0 = S, P_1, \ldots, P_k\}$, and for each $i=1,\ldots, k$ write $L_i = \Aut_{\LL}(P_i)$ and $N_i = \Aut_{\LL}(P_i < S)$ for short. Fix $\Psi \in \atyp^I(\LL)$, and let $\alpha\in\Sigma_k$ denote $\upsilon(\Psi)$, where $\upsilon$ is the homomorphism defined in Lemma \ref{permutation}. 
For each $0\le i \le k$, set $P_i' = P_{\alpha(i)}$, $L_i' \defeq L_{\alpha(i)}$ and $N'_i = N_{\alpha(i)}$. Let $G_0 = G_0' = L_0$, and for $i = 0, \ldots, k-1$ let
\[
G_{i+1} = G_i \ast_{N_{i+1}} L_{i+1} \qquad \text{and} \qquad G'_{i+1} = G'_{i} \ast_{N'_{i+1}} L'_{i+1}.
\]
Notice that $G = G_k = G'_k$. 

\noindent{\bf Step 1.} (Definition of $\gamma$) Let $\Psi_0 = \Psi$.  Since $\Psi_0$ is an isotypical equivalence,  it restricts to an automorphism of $L_0$ and of $S\nsg L_0$. Assume by induction that we have defined, for each $0\le i\le j<k$, an isotypical equivalence $\Psi_i\in\Aut_\typ^I(\LL)$, such that  for each $0<i<j$,  
\begin{enumerate}[(a)]
\item $\Psi_{i} = c_{x_i}\circ \Psi_{i-1}$, for some $x_i\in L_0$, such that $\Psi_{i-1}(P_i) = x_i^{-1}P_i'x_i$,
\item $\upsilon([\Psi_i]) = \alpha$, where $[\Psi_i]\in\Out_\typ(\LL)$ denote the class of $\Psi_i$ modulo the conjugation action of $L_0$, and 
\item $\Psi_i(L_0) = L_0$, and its restriction to $L_i$ and $N_i$ renders the following diagram commutative:
\[\xymatrix{
L_0\ar[d]^{\Psi_i} && N_i\ar[ll]^\incl\ar[rr]_\incl\ar[d]^{\Psi_i} && L_i\ar[d]^{\Psi_i}\\
L_0  && N_i' \ar[ll]_\incl\ar[rr]^\incl && L_i'
}\]
\end{enumerate}
Since $\pp$ is complete, and since $\upsilon(\Psi_{j-1}) = \alpha$, there exists some $x_j\in L_0$, such that  
\[\Psi_{j-1}(P_j) = x_j^{-1} P_j' x_j.\]
Define $\Psi_j \defeq c_{x_j}\circ\Psi_{j-1}$. Conditions a) and b) above are satisfied automatically. Also $\Psi_j$ is clearly an isotypical equivalence of $\LL$, and hence it restricts to an automorphism of $L_0$, of $S\nsg L_0$. Since $\Psi_j(P_j) = P_j'$, its restriction to $L_j$ maps it isomorphically to $L_j'$, and since $\Psi_j$ respects inclusions, its restriction to $N_j$ maps is isomorphically to $N_j'$, so that the corresponding diagram in Condition c) commutes. 

We have now obtained isotypical equivalences $\Psi_i\in\Aut_\typ^I(\LL)$ for $i=0\ldots k$, each of which satisfying the corresponding set of conditions of the form (a), (b) and (c) above. Notice that $\Psi_0$ can be regarded as an isomorphism $\widetilde{\Psi}_0\colon G_0\to G'_0$, and the diagram in Condition (c) shows that $\Psi_i$ induces an isomorphism $\widetilde{\Psi}_i\colon G_i\to G_i'$ for each $i$. Thus in particular $\Psi_k$ induces an automorphism of $G_k = G_k'=G$. Notice that $\Psi_k = c_{x} \circ \Psi$, where $x = x_k\cdot x_{k-1}\cdots x_1\in L_0$ is the product of the elements chosen in our inductive construction of the $\Psi_i$. Hence $\Psi_k$ and $\Psi$ represent the same class in $\Out_\typ(\LL)$. 
Define $\gamma([\Psi])$ to be the class of $\widetilde{\Psi}_k$ in $\Out(G)$. Then $\gamma$ is clearly a well defined function, since by construction is sends conjugation by an element of $L_0$ to an inner automorphism of $G$.
\medskip

\noindent{\bf Step 2.} (Invariance under permutations) We claim that $\gamma$ is invariant under any permutation of the subgroups $P_i$, $i=1\ldots k$. Let  $\sigma\in \Sigma_k$ be an arbitrary permutation. Clearly writing $\pp$ as $\{P_0 = S, P_{\sigma(1)}, \ldots, P_{\sigma(k)}\}$ does not change any of its properties. Performing the  construction of $\gamma$ as in Step 1 with respect to this permuted ordering, one obtains $\Psi'_k = c_y \circ\Psi$, where $y = y_k\cdots y_1$, for appropriate elements $y_j\in L_0$. Consequently $\Psi_k$ constructed in Step 1 and $\Psi_k'$ differ by conjugation by an element of $L_0$, and hence $\widetilde{\Psi}_k$ and $\widetilde{\Psi}'_k$ represent the same class in $\Out(G)$, proving our claim.
\medskip

\noindent{\bf Step 3.} ($\gamma$ is a right inverse for $\omega$) Let $\Psi \in \Aut_{\typ}^{I}(\LL)$ represent a class $[\Psi] \in \Out_{\typ}(\LL)$. Let $x \in L_0$, $\Psi_k = c_x \circ \Psi$ and $\widetilde{\Psi}_k \in \Aut(G)$ be as in the definition of $\gamma$. By construction of $\widetilde{\Psi}_k$, the following holds for each $i = 0, \ldots, k$.
\begin{enumerate}[(1)]

\item $\widetilde{\Psi}_k(P_i) = \Psi_k(P_i) = (c_x \circ \Psi)(P_i)$.

\item $L_i \leq N_{G}(P_i) = \Aut_{\TT^{\bullet}}(P_i)$, and for each $\varphi \in L_i$ we have $\widetilde{\Psi}_k(\varphi) = \Psi_k(\varphi) = (c_x \circ \Psi)(\varphi)$.

\end{enumerate}
In other words, if $\LL_{\pp} \subseteq \LL$ denotes the full subcategory with object set $\pp$, then
$$
\Omega(\widetilde{\Psi}_k)|_{\LL_{\pp}} = (\Psi_k)|_{\LL_{\pp}},
$$
where $\Omega \colon \Aut(G, S) \to \Aut_{\typ}^{I}(\LL)$ is the group homomorphism in Proposition \ref{verybonito2}.

Since $\pp$ is complete, it follows by Lemma \ref{autext2} that $\Psi_k = \Omega(\widetilde{\Psi}_k)$. Since $\Psi_k = c_x \circ \Psi$, we have $(\omega \circ \gamma)([\Psi]) = [\Psi_k] = [\Psi]$. In particular it follows that $\omega$ is an epimorphism.

\medskip

\noindent{\bf Step 4.} ($\gamma$ is a homomorphism)
Let $\Psi, \Psi' \in \Aut_{\typ}^{I}(\LL)$ represent the classes $[\Psi], [\Psi']\in\Out_\typ(\LL)$ respectively, and set $\Phi = \Psi' \circ \Psi$. Thus $[\Phi] = [\Psi']\cdot[\Psi]$. We show now that 
\[\gamma([\Phi]) = \gamma([\Psi']) \cdot \gamma([\Psi]),\]
thus proving that $\gamma$ is a homomorphism.

Perform the construction of $\Psi_i$ and $\Psi_i'$, $0\le i\le k$, as described in Step 1, with respect to the original ordering on $\pp$. Let $\alpha = \upsilon([\Psi])\in \Sigma_k$, and for each $0\le i\le k$ let $\Psi'{}^{\alpha}_i$ denote the $i$-th step in the construction in Step 1 with respect to $\Psi'$ and  the ordering on $\pp$ defined by the permutation $\alpha$. Let $\beta = \upsilon([\Psi'])\in \Sigma_k$, and let $\sigma = \beta\cdot\alpha = \upsilon([\Phi])$.  We claim that for each $0\le i\le k$, the automorphisms $\Phi_i$ can be constructed, with respect to the original ordering on $\pp$, to satisfy the identity $\Psi'{}^\alpha_i\circ\Psi_i =\Phi_i$. 

For $i=0$ the claim is obvious. Thus assume by induction that for all $i<j\le k$ we have constructed $\Phi_i$ such that $\Psi'{}^\alpha_i\circ\Psi_i =\Phi_i$. We construct $\Phi_j$ and show that it satisfies the corresponding identity. By construction $\Psi_j = c_{x_j}\circ\Psi_{j-1}$, where  $x_j\in L_0$ is such that $\Psi_{j-1}(P_j) = x_j^{-1} P_{\alpha(j)}x_j$. Similarly $\Psi'{}_j^\alpha = c_{y_j}\circ\Psi'{}_{j-1}^\alpha$, for some $y_j\in L_0$ such that $\Psi'{}_{j-1}^{\alpha}(P_{\alpha(j)}) = y_j^{-1} P_{\beta\alpha(j)}y_j = y_j^{-1} P_{\sigma(j)}y_j$. Hence 
\[\Psi'{}_j^\alpha \circ\Psi_j = c_{y_j}\circ\Psi'{}_{j-1}^\alpha\circ  c_{x_j}\circ\Psi_{j-1} = c_{y_j}\circ c_{\Psi'{}^\alpha_{j-1}(x_j)}\circ\Psi_{j-1}\circ \Psi'{}_{j-1}^\alpha = c_{y_j\cdot \Psi'{}^\alpha_{j-1}(x_j)}\circ\Phi_{j-1}.\]
Notice that 
\begin{multline*}
(y_j\cdot \Psi'{}^\alpha_{j-1}(x_j))^{-1}\cdot P_{\sigma(j)}\cdot y_j\cdot \Psi'{}^\alpha_{j-1}(x_j) =
\Psi'{}^\alpha_{j-1}(x_j)^{-1}\cdot y_j^{-1}\cdot P_{\sigma(j)}\cdot y_j\cdot \Psi'{}^\alpha_{j-1}(x_j)=\\
\Psi'{}^\alpha_{j-1}(x_j)^{-1}\cdot \Psi'{}_{j-1}^{\alpha}(P_{\alpha(j)})\cdot \Psi'{}^\alpha_{j-1}(x_j) = 
\Psi'{}^\alpha_{j-1}(x_j^{-1}\cdot P_{\alpha(j)}\cdot x_j) = \Psi'{}^\alpha_{j-1}\circ\Psi_{j-1}(P_j).
\end{multline*}
Let $z_j \defeq y_j\cdot \Psi'{}^\alpha_{j-1}(x_j)\in L_0$. Then, by Step 1, $\Phi_j\defeq c_{z_j}\circ\Phi_{j-1}$ satisfies the required properties, as well as the identity
\[\Psi'{}^\alpha_j\circ\Psi_j = c_{z_j}\circ\Phi_{j-1} = \Phi_j.\]
Since this identity has now been shown to hold for all $1\le j\le k$, it follows that $\Psi'{}^\alpha_k\circ\Psi_k =  \Phi_k$. Hence the induced homomorphisms on $G$ satisfy $\widetilde{\Psi}'{}^\alpha_k\circ\widetilde{\Psi}_k = \widetilde{\Phi}_k$. By Step 2, $\widetilde{\Psi}'{}^\alpha_k$ represents the class $\gamma([\Psi'])\in\Out(G)$. Hence our claim follows.
\end{proof}

Propositions \ref{morph2} and \ref{morph1} complete the proof of the main statement of Theorem \ref{main}.
As a corollary of the proof, we obtain Part (i) of Theorem \ref{main}.

\begin{cor}\label{verybonito3}
With the hypotheses of Proposition \ref{morph1}, for each $\Psi \in \atyp^I(\LL)$ there exist a group homomorphism $\Psi' \in \Aut(G, S)$ and an isotypical equivalence $\Omega(\Psi') \in \atyp^I(\LL)$ such that $\Omega(\Psi')$ is naturally isomorphic to $\Psi$ and such that the following square commutes up to homotopy
$$
\xymatrix{
BG \ar[r]^{\mu} \ar[d]_{\Psi'} & B\g \ar[d]^{|\Omega(\Psi')|\pcom} \\
BG  \ar[r]_{\mu} & B\g
}
$$
\end{cor}
\begin{proof}
This follows at once from Step 3 in the proof of Proposition \ref{morph1}.
\end{proof}

\begin{cor}\label{commdiag}
There is a commutative diagram of exact sequences
$$
\xymatrix{
1 \ar[r] & Z(G) \ar@{->}[d]^\cong \ar[r] & N_{G}(S) \ar[r] \ar@{->>}[d] & \Aut(G,S) \ar[r] \ar@{->>}[d] & \Out(G) \ar[r] \ar@{->>}[d] & 1 \\
1 \ar[r] & Z(\g) \ar[r] & \Aut_{\LL}(S) \ar[r] & \atyp^I(\LL) \ar[r] & \otyp(\LL) \ar[r] & 1 \\
}
$$
\end{cor}
\begin{proof}
Commutativity follows by construction of the maps involved. The isomorphism in the first square follows from Corollary \ref{ZGZF}.
\end{proof}

We end with a description of the kernel of $\omega\colon\Out(G)\to \Out_\typ(\LL)$, as stated in Part (ii) of Theorem \ref{main}.

\begin{prop}\label{kernel}
Let $\Aut(G,1_S)\le \Aut(G)$ denote the subgroup of all automorphisms which restrict to the identity on $S$, and let $\Out(G, 1_S)$ denote the quotient by the subgroup of inner automorphisms of $G$ induced by conjugation by elements of $C_{G}(S)$. Then there is a short exact sequence 
\[1\to \Ker(\omega)\to \Out(G, 1_S)\to \higherlim{\oo(\FF^c)}{1}\ZZ/\ZZ_0\to 1.\]
In particular if $p$ is an odd prime then $\Ker(\omega)\cong \Out(G, 1_S)$.
\end{prop}
\begin{proof}
Let $\varphi$ be an automorphism of $G$ which restricts to an automorphism of $S$. Then $\varphi|_S$ is clearly fusion preserving. Hence we obtain a map $\widehat{\omega}\colon\Aut(G, S)\to \Aut_\fus(S)$, and a commutative diagram with exact rows an columns:
\[\xymatrix{
C_{G}(S)/Z(G)\;\ar@{>->}[r]\ar@{>->}[d]&N_{G}(S)/Z(G)\ar@{->>}[r]\ar@{>->}[d] & \Aut_\FF(S)\ar@{>->}[d]\\
\Aut(G, 1_S)\;\ar@{>->}[r]\ar@{->>}[d]&\Aut(G, S)\ar@{->>}[r]^{\widehat{\omega}}\ar@{->>}[d] & \Aut_\fus(S)\ar@{->>}[d]\\
\Out(G,1_S)\;\ar@{>->}[r]&\Out(G)\ar@{->>}[r]^{\bar{\omega}} & \Out_\fus(S)
}\]
By construction of  $\omega$,  the map $\bar{\omega}$ factors as the composite
\[\Out(G)\xto{\omega} \Out_{\typ}(\LL)\to \Out_{\fus}(S).\]
Hence the statement follows from \cite[Proposition 7.2]{BLO3} and the snake lemma. The conclusion for odd primes follows from \cite{LL}.
\end{proof}

%%%%%%%%%%%%%%%%%%%%%%%%%%%%%%%%%%%%%%%%%%%%
%%%%%%%%%%%%%%%%%%%%%%%%%%%%%%%%%%%%%%%%%%%%
%%%%%%%%%%%%%%%%%%%%%%%%%%%%%%%%%%%%%%%%%%%%
%%%%%%%%%%%%%%%%%%%%%%%%%%%%%%%%%%%%%%%%%%%%
%%%%%%%%%%%%%%%%%%%%%%%%%%%%%%%%%%%%%%%%%%%%
%%%%%%%%%%%%%%%%%%%%%%%%%%%%%%%%%%%%%%%%%%%%
%%%%%%%%%%%%%%%%%%%%%%%%%%%%%%%%%%%%%%%%%%%%
%%%%%%%%%%%%%%%%%%%%%%%%%%%%%%%%%%%%%%%%%%%%
%%%%%%%%%%%%%%%%%%%%%%%%%%%%%%%%%%%%%%%%%%%%
%%%%%%%%%%%%%%%%%%%%%%%%%%%%%%%%%%%%%%%%%%%%
%%%%%%%%%%%%%%%%%%%%%%%%%%%%%%%%%%%%%%%%%%%%
%%%%%%%%%%%%%%%%%%%%%%%%%%%%%%%%%%%%%%%%%%%%
%%%%%%%%%%%%%%%%%%%%%%%%%%%%%%%%%%%%%%%%%%%%
%%%%%%%%%%%%%%%%%%%%%%%%%%%%%%%%%%%%%%%%%%%%
%%%%%% SECTION 4%%%  %%%%%%%%%%%%%%%%%%%%%%%%%%%%
%%%%%%%%%%%%%%%%%%%%%%%%%%%%%%%%%%%%%%%%%%%%
%%%%%%%%%%%%%%%%%%%%%%%%%%%%%%%%%%%%%%%%%%%%
%%%%%%%%%%%%%%%%%%%%%%%%%%%%%%%%%%%%%%%%%%%%
%%%%%%%%%%%%%%%%%%%%%%%%%%%%%%%%%%%%%%%%%%%%
%%%%%%%%%%%%%%%%%%%%%%%%%%%%%%%%%%%%%%%%%%%%
%%%%%%%%%%%%%%%%%%%%%%%%%%%%%%%%%%%%%%%%%%%%
%%%%%%%%%%%%%%%%%%%%%%%%%%%%%%%%%%%%%%%%%%%%
%%%%%%%%%%%%%%%%%%%%%%%%%%%%%%%%%%%%%%%%%%%%
%%%%%%%%%%%%%%%%%%%%%%%%%%%%%%%%%%%%%%%%%%%%
%%%%%%%%%%%%%%%%%%%%%%%%%%%%%%%%%%%%%%%%%%%%
%%%%%%%%%%%%%%%%%%%%%%%%%%%%%%%%%%%%%%%%%%%%
%%%%%%%%%%%%%%%%%%%%%%%%%%%%%%%%%%%%%%%%%%%%

\section{A particular case and some examples}

Some Robinson amalgams, especially those where a small number of groups is involved, are exceptionally well behaved, namely the centric transporter category associated to the amalgam turns out to be  isomorphic to the centric linking system associated to the fusion system realised by it. We recall the concept of a  Libman-Seeliger setup $(\pp; \{(L_P, N_P)\}_{P\in\pp})$ for the Robinson amalgam, i.e., a setup where the groups $N_P$ are of the form $\Aut_\LL(S,P)$ for each $P\in\pp$. We restate Proposition \ref{PropB} as 

\begin{prop}\label{only2}
Let $\g = \SFL$ be a $p$-local compact group, let  $(\pp; \{(L_P, N_P)\}_{P\in\pp})$ be a Libman-Seeliger setup, where $\pp$ is a complete fusion controlling family for $\FF$, and let $G$ be be the corresponding Robinson amalgam. Suppose in addition that $\TT^c_S(G) \cong \LL$. Then the homomorphisms
$$
\Omega\colon\Aut(G, S) \xto{} \Aut_{\typ}^{I}(\LL) \qquad \mbox{and} \qquad \omega\colon\Out(G) \xto{} \Out_{\typ}(\LL),
$$
of Propositions \ref{verybonito2} and \ref{morph2} respectively, are isomorphisms.
\end{prop}

Proving the proposition requires some preparation. Fix a $p$-local compact group $\g$ and a Robinson setup with  a fusion controlling family $\pp$ (not necessarily complete), and let $G$ be the associated Robinson amalgam. Let  $\TT^{\bullet}$ be the transporter category associated to $G$ with object set
$$
\Ob(\TT^{\bullet}) = \Ob(\LL^{\bullet}).
$$
For every $P, Q \in \Ob(\TT^{\bullet})$, the following holds.
\begin{enumerate}[(a)]

\item $N_S(P,Q)$ is naturally a subset in $\Mor_{\TT^{\bullet}}(P,Q)$.

\item If $P \leq Q$, then the identity element of $G$ is by definition an element in $\Mor_{\TT^{\bullet}}(P,Q)$, which we denote by $\iota_{P,Q}$.

\end{enumerate}
We define $\Aut_{\typ}^{I}(\TT^{\bullet})$ to be the group of self equivalences $\Phi$ of $\TT^{\bullet}$ such that $\Phi(N_S(P,Q)) \subseteq N_S(\Phi(P), \Phi(Q))$ for all $P, Q \in \Ob(\TT^{\bullet})$, and such that $\Phi(\iota_{P,Q}) = \iota_{\Phi(P), \Phi(Q)}$ whenever $P \leq Q$. Finally, we denote by $\Out_{\typ}(\TT^{\bullet})$ the group of equivalence classes in $\Aut_{\typ}^{I}(\TT^{\bullet})$ under natural isomorphism.

\begin{lmm}\label{ex1}

There are group homomorphisms
$$\Omega_a \colon \Aut(G, S) \to \Aut_{\typ}^{I}(\TT^{\bullet}) \qquad \mbox{and} \qquad \Omega_b \colon \Aut_{\typ}^{I}(\TT^{\bullet}) \to \Aut_{\typ}^{I}(\LL)
$$
such that the homomorphism $\Omega \colon \Aut(G,S) \to \Aut_{\typ}^{I}(\LL)$ in Proposition \ref{verybonito2} factors as $\Omega_b \circ \Omega_a$. Furthermore, $\Omega_a$ is a group monomorphism.

\end{lmm}

\begin{proof}

The homomorphisms $\Omega_a$ and $\Omega_b$ are already described in the proof of Proposition \ref{verybonito2}, and the details of the factorisation of $\Omega$ are left to the reader. To check that $\Omega_a$ is injective, set as usual $\pp = \{P_0 = S, \ldots, P_k\}$, and  $G = L_0 \ast_{N_1} L_1 \ast_{N_2} \ldots \ast_{N_k} L_k$. Then we have
$$
L_i \leq \Aut_{\TT^{\bullet}}(P_i) \qquad i = 0, \ldots, k,
$$
and injectivity follows since $L_0 \cup \ldots \cup L_k$ is a set of generators of $G$.
\end{proof}

\begin{lmm}\label{ex2}

Two elements in $\Aut_{\typ}^{I}(\TT^{\bullet})$ are naturally isomorphic if and only if they differ by conjugation by an element of $\Aut_{\TT^{\bullet}}(S) = N_{G}(S)$.

\end{lmm}

\begin{proof}

This is immediate (compare this to the proof of \cite[Lemma 1.14]{AOV}).
\end{proof}

\begin{lmm}\label{ex3}

The group homomorphisms $\Omega_a$ and $\Omega_b$ in Lemma \ref{ex1} induce group homomorphisms $\omega_a \colon \Out(G) \to \Out_{\typ}(\TT^{\bullet})$ and $\omega_b \colon \Out_{\typ}(\TT^{\bullet}) \to \Out_{\typ}(\LL)$ that factorize the group homomorphism $\omega$ in Proposition \ref{morph2}. Furthermore, $\omega_a$ is a group monomorphism.
\end{lmm}

\begin{proof}

The proof of the first part of the statement is essentially identical to the proof of Proposition \ref{morph2}. The injectivity of $\omega_a$ follows from Lemmas \ref{ex1} and \ref{ex2}: if $[\phi] \in \Out(G)$ is such that $\omega_a[\phi] = 1$, then we have $\Omega_a(\phi) = c_g$, for some $g \in N_{G}(S)$. The injectivity of $\Omega_a$ implies that $\phi = c_g$, and thus $[\phi] = 1$.
\end{proof}

\begin{proof}[Proof of Proposition \ref{only2}]
Suppose that $\TT^c_S(G) \cong \LL$. Then in particular $\TT^{\bullet} \cong \LL^{\bullet}$, where $\TT^{\bullet} \subseteq \TT^c_S(G)$ is the full subcategory with object set $\Ob(\TT^{\bullet}) = \Ob(\LL^{\bullet})$. It follows in this case that both $\Omega_b$ and $\omega_b$ are group isomorphisms, and hence $\omega$ is both injective (by Lemma \ref{ex3}) and surjective (by Proposition \ref{morph1}). The isomorphism $\Aut(G, S) \cong \Aut_{\typ}^{I}(\LL)$ follows from the commutative diagram of Corollary \ref{commdiag}.\end{proof}

We end this section some examples of $p$-local finite groups to which Proposition \ref{only2} applies. To this purpose, we first present a slight generalisation of a result of Libman \cite{Libman}.

\begin{prop}\label{itworks}
Let $\g = \ploc$ be a $p$-local compact group with $Z(S) \cong \Z/p$, and let $\pp = \{P_0 = S, P_1, \ldots, P_k\}$ be a complete fusion controlling family. Let $G$ be the Robinson amalgam associated with a Libman-Seeliger setup for $\pp$. Suppose that the following holds for each $P \in \pp$.
\begin{enumerate}[(i)]

\item $N_P = N_{L_P}(Z(S))$.

\item $N_S(P)$ is nonabelian and $N_P = N_{L_P}(N_S(P))$.

\item $P$ has index $p$ in $N_S(P)$.

\item $L_P$ induces a transitive action on $P/\Phi(P)$, where $\Phi(P)$ is the Frattini subgroup of $P$.

\end{enumerate}
Then, $N_{G}(P) = \Aut_{\LL}(P)$ for all $P \in \FF^c$.

\end{prop}

\begin{proof}

Note that $N_S(P) \nsg N_P$ for all $P \in \pp$, since $N_P \leq \Aut_{\LL}(S)$, and $S \nsg \Aut_{\LL}(S)$. Also, since $P$ is fully normalised for all $P \in \pp$, the subgroup $N_S(P)$ is a Sylow $p$-subgroup of both $N_P$ and $L_P$. Furthermore, since $P \in \pp$ is centric, it must contain $Z(S)$. The statement follows from \cite[Proposition 8.4 (c) and Proposition 8.11]{Libman}, with minor modifications.
\end{proof}

In particular under the hypotheses of Proposition \ref{itworks}, one has the required equivalence $\LL \cong \TT_S^c(G)$ in Proposition \ref{only2}. The following is  a non-exhaustive list of $p$-local compact groups satisfying the conditions of Proposition Proposition \ref{itworks}. 
\begin{cor}

The following $p$-local compact groups satisfy the conditions of \ref{itworks}.
\begin{enumerate}[(i)]

\item The $2$-local finite groups associated to the groups $PGL_2(q)$, where $q$ is some power of an odd prime (such that the $2$-Sylow subgroups are dihedral of order at least $16$).

\item The $2$-local compact groups associated to the compact Lie groups $SO(3)$ and $SU(2)$.

\item The exotic $3$-local finite groups of \cite{DRV}.

\item The exotic $7$-local finite groups of \cite{RV}.

\end{enumerate}

\end{cor}

\begin{proof}

The $2$-local finite group associated to $PGL_2(q)$ is (isomorphic to) the $2$-local finite group associated to $PSL_2(q) \rtimes \Z/2$, where $\Z/2$ is the factor of $\Out(PSL_2(q))$ generated by the ``diagonal'' automorphism (see \cite[Lemma 7.7]{BLO1} for further detail). Regarding the $2$-local compact groups associated to $SO(3)$ and $SU(2)$, the reader is referred to  \cite[Examples 3.7 and 3.8]{Gonza2} for full detail. In the remaining two cases details can be found in the given references. 
\end{proof}

%%%%%%%%%%%%%%%%%%%%%%%%%%%%%%%%%%%%%%%%%%%%
%%%%%%%%%%%%%%%%%%%%%%%%%%%%%%%%%%%%%%%%%%%%
%%%%%% BIBLIOGRAPHY  %%%%%%%%%%%%%%%%%%%%%%%%%%%%
%%%%%%%%%%%%%%%%%%%%%%%%%%%%%%%%%%%%%%%%%%%%
%%%%%%%%%%%%%%%%%%%%%%%%%%%%%%%%%%%%%%%%%%%%
%%%%%%%%%%%%%%%%%%%%%%%%%%%%%%%%%%%%%%%%%%%%

\end{document}